\documentclass{article}
\usepackage[a4paper]{geometry}
\usepackage[english]{babel}
\usepackage{amsthm}
\usepackage[shortlabels]{enumitem}
\usepackage{amscd}
\usepackage{todonotes}
\usepackage{amssymb}
\usepackage{amsmath}
\usepackage{tikz-cd}
\usepackage{graphicx}
\usepackage[colorlinks=true, allcolors=blue]{hyperref}
\usepackage{stackrel}

\usepackage[mode=multiuser,status=draft,lang=english]{fixme}
\fxsetup{theme=colorsig}

\FXRegisterAuthor{MV}{aM}{Matteo}
\FXRegisterAuthor{EP}{aE}{Elena}
\FXRegisterAuthor{MG}{aMG}{Mai}

\newcommand{\aaa}{\mathcal{A} }
\newcommand{\open}[1]{\mathcal{N}^{\mathcal{A}}_{#1}}
\newcommand{\opencorn}[2]{\mathcal{N}^{\mathsf{W}(#1)}_{#2}}
\newcommand{\closedcorn}[2]{\mathcal{C}^{\mathsf{W}(#1)}_{#2}}
\newcommand{\closed}[1]{\mathcal{C}^{\mathcal{A}}_{#1}}

\newcommand{\corn}[1]{(\mathsf{W}(#1),\tau_{#1})}
\newcommand{\wall}[2]{\mathsf{W}(#1,#2)}

\newcommand{\spa}{\sigma^\aaa_P}
\newcommand{\atpa}{(\aaa, \tau^\aaa_P)}

\title{Reflecting compact $T_1$-spaces into bounded distributive lattices}
\author{Mai Gehrke, Elena Pozzan, Matteo Viale}

\theoremstyle{plain}

	\newtheorem{theorem}{Theorem}[subsection]
	\newtheorem{proposition}[theorem]{Proposition}
	\newtheorem{lemma}[theorem]{Lemma}
	\newtheorem{corollary}[theorem]{Corollary}
	\newtheorem{fact}[theorem]{Fact}

	\newtheorem{question}[theorem]{Question}
	\newtheorem{claim}{Claim}
	
\theoremstyle{definition}

	\newtheorem{definition}[theorem]{Definition}
	\newtheorem{notation}[theorem]{Notation}
	\newtheorem{example}[theorem]{Example}

\theoremstyle{remark}
	\newtheorem{remark}[theorem]{Remark}

\DeclareMathOperator{\St}{St}

\DeclareMathOperator{\RO}{RO}

\newcommand{\bool}[1]{\mathsf{#1}}

\newcommand{\pow}[1]{\mathcal{P}\left(#1\right)}

\newcommand{\Cl}[1]{\text{Cl}\left(#1\right)}
\newcommand{\Reg}[1]{\text{Reg}\left(#1\right)}

\newcommand{\cp}[1]{\left( #1 \right)}

\newcommand{\bp}[1]{\left\lbrace #1 \right\rbrace}

\begin{document}
\maketitle

\begin{abstract}
We present a contravariant reflection of the compact $T_1$-spaces with arrows given by \emph{closed} continuous functions into the category of bounded distributive lattices with arrows given by \emph{closed subfit} morphisms. This reflection extends both Stone duality and Isbell's duality between frames and sober spaces for those compact $T_1$-spaces that fall within each of these dualities, that is, respectively, zero-dimensional compact Hausdorff spaces, and compact sober $T_1$-spaces. 
On the topological side, we allow all compact $T_1$-spaces rather than just sober ones and we identify points in these with \emph{minimal} prime filters on \emph{some base}. 
On the lattice side, the shift goes from the notion of frame homomorphism to that of \emph{closed subfit} morphism between bounded distributive lattices (closed subfit morphisms are defined by a natural and first order expressible constraint).
%
The reflection becomes a duality when one restricts on the algebraic side to the complete and compact \emph{subfit} lattices (i.e. compact subfit frames). 
Furthermore, restricting our duality on the topological side to the subcategory of compact $T_2$-spaces with \emph{all} continuous maps, 
we obtain a duality for these with the category of complete, compact and normal lattices, thus recovering a classical result of Cornish. 

We also relate our adjunction to the duality introduced by Maruyama between $T_1$-spaces with continuous maps and a category having as objects a particular type of subfit frames and as arrows a certain type of morphism, of which we give an alternative and explicit 
algebraic characterization. 
\end{abstract}


\section{Introduction}
This\footnote{The third author acknowledges support from the projects
\emph{PRIN 2017-2017NWTM8R
Mathematical Logic: models, sets, computability, PRIN 2022 Models, sets and classifications, prot. 2022TECZJA}, and from GNSAGA.

\textbf{MSC:} \emph{06B05, 06D22, 18F70, 54H12.}  

The authors thank Tristan Bice, and Wieslaw Kubis for many useful indications, especially on the existing bibliography on the topic and the connections of the present results with other works. 
} paper is motivated by the following loosely related questions on general topology:
\begin{enumerate}
\item
Looking at a topological space with the algebraic perspective that duality results (such as Stone's or Isbell's dualities\footnote{Isbell's duality is for us the duality between sober spaces and spatial frames as presented for example in \cite[Cor. IX.3.4]{MacLMoer}; Stone's duality is either the well known one between boolean algebras and Compact $0$-dimensional spaces, or its lesser known generalization between bounded distributive lattices and spectral spaces; see for details the Appendix \ref{Appendix}.}) bring forward,
what is the optimal/right/more efficient algebraic characterization of the notion of point?
\item
The Tychonoff separation property describes the class of Hausdorff spaces which admit a Hausdorff compactification. However the standard definition of this property (e.g. see \cite[Def. D.3.3]{VIABOOKFORCING}) is neither algebraic in nature, nor intrinsic, i.e. to detect whether $(X,\tau)$ is Tychonoff, we need to consider a web of relations between the topology $\tau$, the euclidean topology on $\mathbb{R}$, and the points of $X$. This is in contrast to what happens with other separation properties, whose definitions usually are expressed in terms of lattice-theoretic properties of the spatial frame $\tau$, and of the prime filters of $\tau$ given by the points of $X$.
\end{enumerate}

Note that the answers to the first question given by Stone's and Isbell's dualities are  conflicting:
\begin{itemize}
\item
Stone duality for Boolean algebras identifies points with \emph{ultrafilters} of the Boolean algebra of \emph{clopen} sets,
\item
Isbell duality identifies points with \emph{completely} prime filters on the frame of \emph{open} sets.
\end{itemize}
Abstractly, in both cases, one has a class of bounded distributive lattices of interest (respectively, Boolean algebras and frames) and one attaches to each member of the class a topological space whose points are a certain family of prime filters.

Let us consider the class of algebras for which both notions of point make sense, i.e. the complete Boolean algebras. When a complete Boolean algebra is atomless, it has no points according to Isbell's notion, while it has plenty of points according to Stone's notion. Can we reconcile these opposing views?

It turns out that the answer to the second question, on the Tychonoff separation property, pioneered by Frink and Cornish, based on ideas by Wallman (see \cite{Cornish1972NormalL,frink,WallmanComp37}) gives useful insights on how to recompose this opposition.

A starting point of Frink and Cornish is the observation that key properties of a topological space $(X,\tau)$ are captured by \emph{some} base, $P$, for the topology $\tau$ (depending on the context, $P$ could be the family of complements of $0$-sets, or the family of clopen sets, or the family of regular open sets, etc) and by the ensuing families of filters $F_{x,P}$ given by the open neighboorhoods of $x$ that belong to $P$. Frink, in \cite{frink}, characterizes Tychonoff spaces as those spaces admitting a \emph{normal} base for the closed sets (i.e. a base for the closed sets which satisfies exactly those algebraic properties of the lattice of $0$-sets which are essential in the proof of the Stone-\v{C}ech compactification theorem, see Definition~\ref{def:normalbasis}). Subsequently Cornish, in \cite{Cornish1972NormalL}, identifies the first order axiomatizable class of bounded distributive lattices whose spaces of \emph{minimal} prime filters $F$ are compact Hausdorff (the normal lattices, see Definition~\ref{def:normallattice}).\footnote{The topology of open sets of such a space is generated by the sets $N_p$ consisting of those $F$ with $p\in F$.} 


Here, we will present what we consider a satisfactory approach to our first question while encompassing key portions of both Isbell's and Stone's duality; this approach has the further benefit of allowing us to extend the results of Cornish to the class of compact $T_1$-spaces. Our 
answers to the first question are: 
\begin{enumerate}[(a)]
\item
Points are \emph{minimal} prime filters on lattices (as in Wallman/Frink/Cornish); this gives an efficient characterization of the notion of point for $T_1$-spaces (see Proposition \ref{prop:separationproperties}(4)).
\item
The interesting lattices to be considered are those which are isomorphic to a base of some $T_1$-topological space (i.e. the subfits of Isbell~\cite{Isbell}, see Def. \ref{def:subfit} and Cor. \ref{cor:subfit}).
\item Compactness plays a crucial role and allows to reconcile Isbell and Cornish's notion of point; in particular on compact subfit frames Cornish points are Isbell points (see the analysis of the comparison between Cornish and Isbell's notions we carry out in Section \ref{subsec:cornishvsisbell}).
\end{enumerate}

Our first main result (Thm. \ref{thm:adj-dual}) is the identification of a contravariant reflection between:
\begin{itemize} 
\item
the category having as objects the compact $T_1$-spaces and as arrows the \emph{closed} continous maps on them, and
\item
the category $\bool{SbfL}_\bool{closed}$ having as objects the subfits, i.e. the (first order axiomatizable) class of bounded distributive
lattices which are isomorphic to a base for some $T_1$-topological space, and as arrows \emph{closed subfit} morphisms, a class of morphisms defined by a natural (and first order expressible) constraint.\footnote{Note that the reflection is well defined also on the larger category given by bounded distributive lattices with \emph{subfit} morphisms.} 
\end{itemize}


This contravariant reflection embeds the compact $T_1$-spaces with continuous closed maps as a reflective subcategory of $\bool{SbfL}_\bool{closed}$; furthermore this adjunction restricts to a duality on the subfamily of compact complete lattices in $\bool{SbfL}_\bool{closed}$.
As a byproduct of the above, we can generalize to a duality (i.e. we extend to arrows) Cornish's result of \cite{Cornish1972NormalL} which identifies compact Hausdorff spaces with compact and complete normal lattices.
One of our adjuncts assigns to a lattice in $\bool{SbfL}$ its space of minimal prime filters, and ---when applied to a Boolean algebra--- outputs its Stone-space (hence it generalizes to a much larger class of lattices the functor which Stone's duality defines on the class of Boolean algebras), while our other adjunct is the restriction to compact $T_1$-spaces of the forgetful functor $(X,\tau)\mapsto\tau$ from topological spaces to spatial frames used in Isbell's duality.\footnote{Note that in compact $T_1$-spaces the minimal prime filters on the topology are always completely prime (see Thm. \ref{lem:primecomplprime}); however in general completely prime filters on a spatial frame need not be minimal prime filters. In particular applying both functors of the Isbell's duality to a compact $T_1$-spaces $(X,\tau)$ may not yield back the original space. On the other hand this will be the case when applying both functor of our adjunction to a compact and $T_1$ $(X,\tau)$.}

Another interesting corollary of our adjunction is a lattice-theoretic reformulation of the Stone-Cech compactification theorem (Thm. \ref{thm:stonecech-lattices}). It applies to lattices which are duals of compact $T_1$-spaces (i.e. the compact complete subfits), and carves inside them their largest normal sublattice.

Finally we connect our adjunction to a duality between $T_1$-spaces and subfit frames 
introduced by Maruyama in \cite{maruyama}. While the Maruyama duality functor on objects is conspicuously defined, Maruyama describes implicitly the relevant arrows in the category of lattices, as those morphisms whose dual maps happen to be continuous; in the present paper we give an explicit algebraic characterization of these morphisms (Prop. \ref{prop:charsubfitmorphism}).
Finally we show that when we restrict Maruyama's duality to compact subfit frames, it exactly coincides with our duality on objects; this occurs because compactness allows to identify Cornish's notion of point which is at the heart of our definition of the right adjunct from subfit lattices to compact $T_1$-spaces, with Isbell's notion of point which is at the heart of Maruyama's definition of the right adjunct in his duality between subfit frames and $T_1$-spaces (provided -as it occurs in Maruyama's duality- that one restricts the space of Isbell's points to the family of \emph{minimal} completely prime filters), see Proposition \ref{lem:primecomplprime}.
%

The paper is structured as follows:

\begin{itemize}
\item  Section \ref{Preliminaries-background} presents some known results and backgorund material. Specifically:
\ref{subsec:topspacesgivenbyfilters}
shows that any $T_0$-space $(X,\tau)$ can be presented as a family of filters on a partial order $P\subseteq \tau$ which is a base for $\tau$;
\ref{subsec:separation}
uses this characterization to reformulate the $T_1$ and $T_2$-separation properties of a space $X$ in terms of order-theoretic properties of the poset $P$ and the family of filters on $P$ used to present $X$.

\item Section \ref{sec:filtersonboundeddistributivelattices} develops a lattice-theoretic reformulation of several topological properties among which: compactness and normality.
It also presents a general machinery (which rephrases the Wallman compactification process, stemming from the seminal work \cite{WallmanComp37}, in the terminology developed here) attaching to a bounded distributive lattice $P$ the compact $T_1$-space $(\bool{W}(P),\tau_P)$, where: $\bool{W}(P)$ is given by the minimal prime filters $F$ on $P$, the topology $\tau_P$ is generated by the sets $N_p$ consisting of those $F$ with $p\in F$. The results of Frink and Cornish appearing in \cite{Cornish1972NormalL,frink} are then presented in this framework.

\item
Section \ref{sec:adjunctions}
presents the first main results, i.e the adjunction between compact $T_1$-spaces with closed continuous maps and $\bool{SbfL}_\bool{closed}$ (Thm. \ref{thm:adj-dual}), and the lattice-theoretic reformulation of the Stone-\v{C}ech compactification theorem (Thm. \ref{thm:stonecech-lattices}).
\item
Section \ref{sec:relation with pre-existence literature} 
relates the results of the preceding sections to the existing literature, specifically:
\begin{itemize}
\item
We first compare the Isbell and Cornish notions of point and analyze the different compactification of a $T_0$-space $(X,\tau)$ given by its sobrification, or by applying to it the Wallman operator we describe in Section \ref{sec:filtersonboundeddistributivelattices}, or by the compactification induced by applying the Stone duality for distributive lattices on some base for $\tau$ closed under finite unions and intersections.
\item
Next we give an explicit algebraic characterization of the notion of continuous map between $T_1$-spaces and use it to outline the connections of the results of Section \ref{sec:adjunctions} with the mentioned duality of Maruyama, and also with a duality appearing in \cite{bice2020wallmandualitysemilatticesubbases} between compact $T_1$-space and a certain category of semilattices with arrows given by a specific type of binary relations (which may not be functional).
We also compare our results to those appearing in \cite{bice2020wallmandualitysemilatticesubbases}.
\end{itemize}
\item
The Appendix \ref{Appendix} includes the basics on the Stone and Isbell dualities needed in Section \ref{sec:relation with pre-existence literature}.

\end{itemize}

The paper is mainly the work of the second author under the supervision of the third author. The first author contributed at a late stage, with important inputs on the existing literature and on how to connect the present results to other well known dualities. 
More precisely:
the second and third author contributed the most part of Sections \ref{sec:filtersonboundeddistributivelattices} and \ref{sec:adjunctions} (note that of these two sections only Section \ref{sec:adjunctions} contains original results), while the results in Section  \ref{sec:relation with pre-existence literature}  are essentially due to the first and second author.

\section{Preliminaries and background material}\label{Preliminaries-background}

In sections \ref{subsec:topspacesgivenbyfilters} and \ref{subsec:separation}, we recall several folklore results characterizing $T_0$-topological spaces in terms of preorders that describe a topological base for the space. We also discuss the relationships between the combinatorial properties of these preorders and the separation and compactness properties of the associated spaces. Although most of these results are well known, we include them for completeness and to facilitate their use in subsequent sections.

\subsection{Topological spaces given by filters on preorders}
\label{subsec:topspacesgivenbyfilters}

\begin{definition}
\label{def:densefamilyfilters}
    Let $P$ be a preorder with a maximum $1_P$. A \emph{prefilter} is a non-empty subset of $P$ such that all its finite subsets have a lower bound in \(P\) that is not a minimum of \(P\). A \emph{filter} $F$ of $P$ is a non-empty subset of \(P\) such that it is closed above in the order and all its finite subsets have a lower bound in \(F\).
A family $\aaa$ of filters in $P$ is \emph{dense in $P$} if for every $p \in P $ that is not a minimum,\footnote{Note that $p$ could be a minimal element (this will correspond to an isolated point).}  there is $F \in \aaa$ such that $p \in F$.
\end{definition}

Given a family of filters $\aaa$ on a preorder $P$, we can define a topological space $\atpa$ whose points are the filters in $\aaa$ and whose topology is generated by the sets 
\[
\open{p}=\bp{F \in \aaa: p \in F}
\]
for every $p \in P$. We also define the corresponding 
subbase for the closed sets  $\spa=\bp{ \closed{p}: p \in P}$ where 
\[
\closed{p}=\aaa \setminus \open{p}
\]

\begin{fact}
\label{fact:omeomspacespreorders}
Let $(X, \tau)$ be a topological space and $\mathcal{B}$ be a basis of non-empty open sets for $\tau$. Let $P=(\mathcal{B}, \subseteq)$ and $\aaa_X=\bp{F_x : x \in X}$ where $F_x=\bp{ U \in \mathcal{B}: x \in U}$. Then, the map 
\begin{align*}
       i_{(X,\tau)}^\mathcal{B}: (X, \tau) &\to (\aaa_X,\tau^{\aaa_X}_P)\\
       x &\mapsto F_x
\end{align*}
is open, continuous and surjective. Moreover, $i_{(X,\tau)}^\mathcal{B}$ is a homeomorphism if and only if $(X,\tau)$ is $T_0$.
\end{fact}

\begin{proof} 
Note that $i_{(X,\tau)}^\mathcal{B}[U]=\bp{F_x: x \in U}=\mathcal{N}_U^{\aaa_X}$ for every $U \in P$, thus the map is open. Similarly, $i_{(X,\tau)}^\mathcal{B}$ is continuous: 
\[
(i_{(X,\tau)}^\mathcal{B})^{-1}[\mathcal{N}_U^{\aaa_X}]=\bp{x \in X: F_x \in\mathcal{N}_U^{\aaa_X}}=\bp{x \in X: U \in F_x}=U.
\] 
Clearly,  $i_{(X,\tau)}^\mathcal{B}$ is injective if and only if $(X,\tau)$ is $T_0$. 
\end{proof}

Given a preorder $(P, \le)$  and $x, y \in P$, we say that $x$ and $y$ are \emph{incompatible} (denoted by $x \perp y$) if:
\[
\forall z \in P\,[ (z \leq x \wedge z \leq y) \implies z \mathrm{ \ is \ a\ minimum \ of\ } P] 
\]

A preorder $(P, \le)$ is \emph{separative} if for all $x$ and $y$ in $P$: 
\begin{equation*}
        x \not \leq y \implies \exists z \in P\,( z \leq x \wedge z \perp y) 
    \end{equation*} 

Clearly, if $p$ and $q$ are incompatible elements of $P$, then there is no refinement of $p$ and $q$, and therefore there is no filter on $P$ that contains both $p$ and $q$.

The regularization, $\Reg{A}$, of a set $A$ in a topological space $(X,\tau)$ denotes the interior of the closure of $A$; $A$ is regular open if $A=\Reg{A}$.
$\RO(X,\tau)$ is the complete Boolean algebra of regular open subsets of $(X,\tau)$.

Given a preorder $(P,\le)$ and $X\subseteq P$, 
\[
\downarrow X=\bp{q:\exists p\in X\, q\le p}
\]
is the downward closure of $X$ (for notational simplicity
$\downarrow p=\downarrow\bp{p}$ for $p\in P$).
The downward topology $\tau_P$ on a partial order $P$ is given by the downward closed subsets of $P$.

\subsection{Separation}
\label{subsec:separation}

\begin{proposition}
\label{prop:separationproperties}
Let $P$ be a preorder with a unique maximum and minimum and let $\aaa$ be a dense family of proper filters on $P$. 
Then: 
\begin{enumerate}
        \item \label{prop:separationproperties-1}
$\atpa$ is a $T_0$ topological space;

        \item \label{prop:separationproperties-2}
$\bp{\open{p}: p \in P}$ is a basis for $(\aaa, \tau^\aaa_P)$;

        \item \label{prop:separationproperties-3}
for each $F \in \aaa$, $\bp{\open{p}: p \in F}$ is a neighborhood basis for F;

        \item \label{prop:separationproperties-4}
$\atpa$ is $T_1$ if and only if $\aaa$ is an antichain in   $(\pow{P}, \subseteq)$;

        \item \label{prop:separationproperties-5}
for each $F \in \aaa$ and $p \in P$, $F \in \Cl{\open{p}}$ if and only if $q$ and $p$ are compatible in $P$ for all $q \in F$. In particular, $F \in \aaa \setminus \Cl{\open{p}}$ if and only if there exists $q \in F$ such that $q \perp p$;

        \item \label{prop:separationproperties-6}
$\atpa$ is Hausdorff if and only if for 
distinct $F, G \in \aaa$, $F \cup G$ is not a prefilter on $P$; 
%
%

    \end{enumerate}
\end{proposition}

\begin{proof} $ $
    \begin{enumerate}
        \item[\ref{prop:separationproperties-1}] If $F \neq G$ are filters of $\aaa$, then we can assume, without loss of generality, that there exists an element $p$ of $P$ such that  $p \in F$ and $p \notin G$. Then $F \in \open{p}$ and $G \notin \open{p}$.

         \item[\ref{prop:separationproperties-2}]  Assume $\open{p_1} \cap \dots \cap \open{p_k} \neq \emptyset$. Let $F \in \open{p_1} \cap \dots \cap \open{p_k} $, i.e. $p_1,\dots,p_k \in F$. Since $F$ is a filter, there exists $q \in F$ such that $q \leq p_i$ for $i=1, \dots, n$. Then  $\open{q} \subseteq \open{p_1} \cap \dots \cap \open{p_k}$.

         \item[\ref{prop:separationproperties-3}]  Let $F \in \aaa$ and $A$ be an open set of $\atpa$ such that $F \in A$. By item \ref{prop:separationproperties-2}, $A = \bigcup \bp{\open{r} : r \in R}$ for some non-empty set $ R \subseteq P$. Therefore, $F \in \open{r}$ and $\open{r} \subseteq A$ for some $r \in R$.

         \item[\ref{prop:separationproperties-4}]   $F \subseteq G$ if and only if $G \in \open{p}$ for every $p \in F$.

         \item[\ref{prop:separationproperties-5}]    Since $\bp{\open{q}: q \in F}$ is a neighborhood basis for $F$, $F \in \Cl{\open{p}}$ if and only if $\open{q} \cap \open{p} $ is non-empty for all $q \in F$. If $F \in \Cl{\open{p}}$ and $q \in F$, then there exists a filter $H \in \open{p} \cap \open{q}$, i.e. $p,q \in H$. Hence, $p$ and $q$ are compatible. Conversely, suppose that $\open{p} \cap \open{q}=\emptyset$ for some $q \in F$. By contradiction, assume that $p,q\geq r$ for some $r \in P$.
         Since $\aaa$ is dense, there exists $H\in\aaa$ such that $r \in H$, but this implies that $p,q \in H$, i.e. $H \in \open{p} \cap \open{q}$.

         \item[\ref{prop:separationproperties-6}]  We have already proved that for any $p,q \in P$, $\open{p} \cap \open{q}=\emptyset$ if and only if $p \perp q$. Let $F\neq G$ in $\aaa$. $F \cup G$ is not a prefilter on $P$ if and only if there exist $p \in F$ and $q \in G$ such that $p \perp q$ if and only if $F \in \open{p}$, $G \in \open{q}$ and $\open{p} \cap \open{q}=\emptyset$.

\end{enumerate}
\end{proof}
%
%
%

\section{Filters on bounded distributive lattices}
\label{sec:filtersonboundeddistributivelattices}

In this section our focus is the investigation of algebraic conditions on a bounded distributive lattice $P$ that characterize topological properties of the spaces $(\mathcal{A},\tau^{\mathcal{A}}_P)$ induced by families of filters on $P$. We aim to characterize algebraically which bounded distributive lattices can be represented as a base of open sets of some $T_1$-space (e.g. subfits), and which bounded distributive lattices give rise to bases for compact Hausdorff spaces.
A key notion for this analysis is given by that of minimal prime filter.

Recall that:
\begin{itemize}
\item
An \emph{upper semi-lattice} $(P,\leq,\vee)$ is a preorder with a binary supremum operation $(a,b)\mapsto a\vee b$, while a \emph{lower semi-lattice} $(Q,\leq,\wedge)$ is a preorder with a binary infima operation $(a,b)\mapsto a\wedge b$.
\item
A \emph{lattice} $(P,\leq,\vee,\wedge)$ is both an upper and lower semi-lattice. 
\item 
A lattice $(P,\leq,\vee,\wedge,0,1)$ is \emph{bounded} if it has a maximum $1$ and a minimum $0$, while it is \emph{distributive} if it satisfies the distributivity laws holding for $\cap$ and $\cup$ in a field of sets (with $\cap$  and $\cup$ interpreting $\wedge$ and $\vee$).
\end{itemize}

 \begin{definition}
 \label{def:primefilter}
     A filter $F$ on an upper semi-lattice $(P,\leq,\vee)$ is \emph{prime} if, for every $p,q \in P$, $p \vee q\in F$ implies $p \in F$ or $q \in F$. A proper filter $F$ is a \emph{minimal prime} filter if it is prime and does not properly contain any non-empty prime filter. 
 
 An ideal $I$ on a lower semi-lattice $(Q,\leq,\wedge)$ is \emph{prime} if, for every $p,q \in Q$, $p \wedge q\in I$ implies $p \in I$ or $q \in I$. A proper ideal $I$ is \emph{maximal} if there exists no other proper ideal $J$ with $I$ a proper subset of $J$.
 \end{definition}
 
Families of prime filters (ideals) on an upper (lower) semi-lattice $P$ reflect in unions (intersections) the join (meet) of elements of $P$. In other words:  

\begin{fact}
\label{fact:un&intprimefilters}
Assume $P$ be an upper semi-lattice and the filters in $\aaa$ be prime. Then: \begin{align*}
    \open{p} \cup \open{q} = \open{p \vee q} 
\end{align*} for every $p,q \in P$.

Furthermore, if $P$ is also a lower semi-lattice, then: 
\begin{align*}
    \open{p} \cap \open{q} = \open{p \wedge q} 
\end{align*} for every $p,q \in P$.
\end{fact}

\begin{proof}
The first equality follows immediately from the fact that the filters of $\aaa$ are prime, while the second one follows from the fact that filters are closed upwards and with respect to finite meets. 
\end{proof}

Clearly, a dual fact holds for lower semi-lattices.

It is important to emphasize why the hypothesis of working with   bounded distributive lattices and prime filters is not particularly restrictive.

Let us consider a $T_0$-topological space $(X, \tau)$ and  a basis $\mathcal{B}$ of open sets for $\tau$. Consider the preorder $P=(\mathcal{B}, \subseteq)$ and $\aaa_X=\bp{F_x : x \in X}$, where $F_x=\bp{ U \in \mathcal{B}: x \in U}$, then (as already remarked) $(\aaa_X, \tau^{\aaa_X}_P)$ is homeomorphic to $(X, \tau)$. If we assume $\mathcal{B}$ is closed under finite unions and intersections, with $X,\emptyset \in \mathcal{B}$, $P$ is a bounded distributive lattice.
Moreover, it is worth noting that in this context, the filters in $\aaa_X$ are always prime: if $U \cup V \in F_x$ with $U, V \in P$, then $x \in U$ or $x \in V$, implying $U \in F_x$ or $V \in F_x$.

\begin{lemma} 
\label{lemma:idealsandfilters}

Let $L$ be a  bounded distributive lattice. Then 
\begin{align*}
        -^C: \bp{J \subseteq L : J \mathrm{ \ is\ a\ proper\ prime\ ideal\ }} &\to \bp{F \subseteq L : F \mathrm{ \ is \ a\ proper\ prime\ filter\ }}\\
        J &\mapsto \bp{p \in L: p \notin J}
\end{align*}
is a bijective map whose inverse is $-^C$ itself.  
\end{lemma}
For a proof see \cite[Lemma 3.9]{gehrke2024topological}.

\begin{corollary}
\label{cor:idealsandfilters}
    Let $L$ be a  bounded distributive lattice. Then  $J$ is maximal ideal of $L$ if and only if $J^C$ is a minimal prime filter.
\end{corollary}
\begin{proof}
    Obvious.
\end{proof}

Consequently:
\begin{fact}\label{fact:idealsandfilters-1}
Let $L$ be a  bounded distributive lattice. Then $F$ is a minimal prime filter on $L$ if and only if for all $p\in F$ there is $q\not\in F$ such that $p\vee q=1_L$.
\end{fact}
\begin{proof}
Note that $F$ is a minimal prime filter if and only if $F^C$ is a maximal prime ideal by Lemma \ref{lemma:idealsandfilters}. Therefore $p\in F$ if and only if $p\not\in F^C$ if and only if for some $q\in F^C$ $p\vee q=1_L$ if and only if for some $q\not\in F$ $p\vee q=1_L$.
\end{proof}

\subsection{Subfit lattices}
\label{subsec:subfitlattices}
We revisit the definition of subfit frame as given in \cite{picado2011frames}, and extend it to the context of bounded distributive lattices.

\begin{definition} 
\label{def:subfit}
A bounded distributive lattice $P$ \ is \emph{subfit} if for every distinct elements $a,b \in P$ there exists $c \in P$ such that $a \vee c= 1_P$ and $b \vee c \neq 1_P $ or vice versa.
\end{definition}

\begin{remark} \label{rem:subfitness}
Subfitness entails the following (seemingly stronger) property:
\begin{equation}\label{eqn:subfitness}
\text{\it{For every  $a,b \in P$ such that $a \not \leq b$ there exists $c \in P$ such that $a \vee c= 1_P$ and $b \vee c \neq 1_P $.}}
\end{equation}

To see this assume $P$ is subfit according to \ref{def:subfit}, and
let $a \nleq b$ in $P$. Then $a'=b\wedge a< a$; in particular $a'\neq a$, thus there is a $c\in P$ with
\[
(a' \vee c=1 \text{ and } a \vee c < 1) \text{ or } (a' \vee c<1 \text{ and }a \vee c=1).
\]
Since $a'< a$, the leftside disjunct cannot be the case, so
\[
a' \vee c<1 \text{ and }a \vee c=1
\]
holds.
Now it follows that
\[
b\vee c = (b\vee c)\wedge 1 = (b\vee c)\wedge(a \vee c)
= (b \wedge a)\vee c = a’ \vee c <1
\]
And \ref{eqn:subfitness} holds.
\end{remark}

Here is an informative characterization of subfitness:
\begin{fact}
\label{fact:charsubfitness}
A bounded distributive lattice $L$ is subfit if and only if for all $a\in L$:
\begin{equation}\label{eqn:subfitness2}
\uparrow a = \bigcap \bp{ F : F \text{ is a minimal prime filter and } a\in F}.
\end{equation}
\end{fact}
\begin{proof}
First assume $L$ is subfit: clearly $\uparrow a \subseteq \bigcap \bp{ F : F \text{ is a minimal prime filter and } a\in F}$ holds.
Now let $b\not\in\uparrow a$, i.e. $a\nleq b$. Then there is $c$ with 
$a\vee c=1$ and $b\vee c\neq 1$. Let F be a minimal prime filter with $b\vee c \not\in F$ (it exists by a basic application of Zorn's Lemma). Then $a\in F$ and $ b\not\in F$, thus the converse inclusion holds as well.

Conversely assume \ref{eqn:subfitness2} holds. Given $a\nleq b$, find (by \ref{eqn:subfitness2}) $F$ minimal prime with $a\in F$ and $b\notin F$.
Then there must be $c\in F^c$ such that $a\vee c=1$ (since $F^c$ is a maximal ideal). Since $b,c\in F^c$ we also have that $b\vee c<1$ (since $F^c$ is a prime ideal). Since $a\nleq b$ are chosen arbitrarily, $L$ is subfit.
\end{proof}

\begin{definition}
\label{def:largefamily}
Let $P$ be a preorder and $\aaa$ be an antichain of prime filters on $P$. 

We say that $\aaa$ is \emph{large} for $P$ if for every two distinct elements $p,q \in P$ there exists $F \in \aaa$ such that $p \in F$ and $q \notin F$ or vice versa. 
\end{definition}

Observe that if \(P\) has a unique minimum $0_P$, any large family of filters is dense on the partial order $P^+=P\setminus\bp{0_P}$.

\begin{fact}
\label{fact:neighborhoodfiltersarelarge}
Assume $(X,\tau)$ is a $T_0$-topological space and $P$ is a base for $\tau$.
Let $\aaa_X$ be the family of filters $F_x=\bp{p\in P:x\in p}$ for all $x \in X$.

Then $\aaa_X$ is large for $P$.
\end{fact}

\begin{proof}
Given $p\neq q$ in $P$, we may assume without loss of generality that there exists $x \in X$ such that $x \in p\setminus q$.
Then $p\in F_x$ and $q\not\in F_x$.
\end{proof}

\begin{proposition}
\label{prop:charsubfit}
Let $P$ be a bounded distributive lattice.
Then, $P$ admits a large  family $\mathcal{A}$ of minimal prime filters on $P$  if and only if $P$ is subfit.
\end{proposition}
\begin{proof}
The right-to-left implication follows straightforwardly from Fact \ref{fact:charsubfitness}.
For the converse implication, suppose that \( P \) admits a large family of minimal prime filters. If \( a \neq b \), without loss of generality, we may assume that there exists a minimal prime filter \( F \) such that \( a \in F \) and \( b \notin F \). Hence, by Fact \ref{fact:idealsandfilters-1}, there exists an element \( c \notin F \) such that \(a \vee c = 1_P \) and \( b \vee c \neq 1_P\).
\end{proof}

\begin{corollary}
\label{cor:subfit}
Let $P$ be a bounded distributive lattice. Then the following are equivalent:
\begin{enumerate}
\item
$P$ is subfit;
\item
For some family $\aaa$ of minimal prime filters on $P$, the map $p\mapsto \open{p}$ is an injective homomorphism of bounded lattices.
\end{enumerate}
\end{corollary}

\begin{proof}
Obvious.
\end{proof}

Subfit lattices provide a sufficient framework for representing \( T_1 \)-topological spaces. Specifically, if we consider a base for such a space that includes the complements of singletons and forms a bounded lattice, it will be subfit. Moreover, the associated neighborhood filters form a large  family of minimal prime filters. Actually we will see that we can characterize subfits as those bounded distributive lattices that are isomorphic to a base of a compact $T_1$-space.

\subsection{Compactness}
\label{subsec:compactness}

\begin{theorem}
 (Cornish \cite{Cornish1972NormalL})
    \label{thm:compactT1} 
    Let $P$ be a bounded distributive lattice and 
    \[
    \aaa=\bp{F \subseteq P: F\mathrm{ \ is\ a\ minimal\ prime\ filter \ on \ }P}
    \]
    Then $\atpa$ is a compact $T_1$-space. 
\end{theorem}

We give our own proof of the theorem since it highlights ideas we will use later on to prove original results (our focus is on minimal prime filters of open sets, in \cite{Cornish1972NormalL} Cornish works with the dual notion of maximal filter of closed sets).
\begin{proof}
$\atpa$ is $T_1$
by Proposition \ref{prop:separationproperties}, since $\aaa$ is an antichain under inclusion.

    Let $\bp{\open{p}: p \in Q }$ be a family of basic open sets of $\atpa$ so that, for all finite $Q'\subseteq Q$, the subfamily $\bp{\open{p}: p \in Q' }$ does not cover $\aaa$. We need to show that $Q$ is not a cover of $\aaa$. That is, there is a minimal prime filter $G$ of $P$ so that $G\not\in\open{p}$, or equivalently $p\not\in G$, for all $p \in Q$.  By Corollary \ref{cor:idealsandfilters}, this is equivalent to showing the existence of a maximal ideal $G^C$ that contains $Q$. To produce the latter we use Zorn's Lemma.  
Let
\begin{equation*}
    \mathcal{Z}=\bp{Z\subseteq L: Q\subseteq Z \ \mathrm{ and } \  \bigvee Z' \neq 1_P \mathrm{\ for\ every\ finite\ } Z'\subseteq  Z}
\end{equation*}
ordered by the inclusion.

\begin{description}[font=\normalfont\itshape]
 \item [$\mathcal{Z} \neq\emptyset$.] If $Q' \subset Q$ is finite, then $Q'$ is not a cover of $\aaa$, that is, $\bigvee Q'\neq 1_P$. Therefore, every finite join of elements from $Q$ is different from $1_P$. Hence, $Q \in \mathcal{Z}$.

    \item [Every chain of $\mathcal{Z}$ admits an upper bound.] Assume that $\bp{Z_\alpha: \alpha \in A} \subseteq \mathcal{Z}$ is a chain. $V=\bigcup_{\alpha \in A} Z_\alpha $ is in $\mathcal{Z}$: if $z_1, \dots, z_n \in V$ then there exists $\beta \in A$ such that $z_1, \dots, z_n \in Z_\beta$, thus $\bigvee_{i=1}^n z_i \neq 1_P$.
\end{description}

By Zorn's Lemma, there exists a maximal element $J$ of $\mathcal{Z}$. An argument left to the reader shows that $J$ is a maximal ideal of $P$ which clearly contains $Q$. 
\end{proof}

\subsection{Normal lattices versus compact Hausdorff spaces}
\label{subsec:normallattices}

Recalling the notion of normal frame given in \cite{picado2011frames}, we generalize it to  bounded distributive lattices.

\begin{definition} 
\label{def:normallattice}
A preorder $P$ \ is a \emph{normal lattice} if the following conditions are met:
\begin{enumerate}
    \item \label{def:normallattice1}
$P$ is a bounded distributive lattice; 
    \item \label{def:normallattice2}
for every distinct elements $p,q \in P$  such that $p \vee q=1_P$, there exist $r,s \in P$ such that $r \wedge s=0_P$ and $r\vee p=s \vee q=1_P$.
\end{enumerate} 
\end{definition}

\begin{theorem}
\label{thm:compactT2}

     Let $P$ be a normal lattice and 
     \[
     \aaa=\bp{F \subseteq P: F\mathrm{ \ is\ a\ minimal\ prime\ filter }}
     \]
     Then $\atpa$ is a compact Hausdorff space. 
\end{theorem}

\begin{proof}
    We only need to prove that $\atpa$ is Hausdorff. By Proposition \ref{prop:separationproperties}, this is equivalent to proving that the union of two minimal prime filters $F,G$ is not a prefilter.

    Without loss of generality, we can suppose there exists an element $p \in F$ such that $p \notin G$. Since $F$ is a minimal prime filter, $F^C$ is a maximal ideal of $P$, by Corollary \ref{cor:idealsandfilters}. Hence, there exists some $q$ such that $p \vee q=1_P$ and  $q \in F^C$, i.e. $q \notin F$. Since $P$ is a normal lattice, there exist some $r,s \in P$ such that \begin{align*}
        p \vee r=1_P && q \vee s=1_P && r \wedge s=0_P
    \end{align*}
    Clearly, $1_P \in F$ and $q \notin F$ imply that $s \in F$. Similarly, $1_p \in G$ and $p \notin G$ imply that $r \in G$. Hence, $r,s \in F \cup G$ and, consequently, $F \cup G$ is not a prefilter. 
\end{proof}

A similar result  can be found in a different form in \cite[Thm. 7.3]{Cornish1972NormalL}. However, our result is more general, as it does not require the assumption of a semi-complemented lattice.

\subsection{The Wallman compactification operator}
\label{subsec:wallman-frinkcompactification}

We provide an overview of Frink's article \cite{frink}, focusing on his introduction of a characterization for Tychonoff spaces and relating it to the results of the previous sections. 

\begin{definition}
\label{def:weakhausdorffcompactification}
Let $(X,\tau)$ be a topological space. A \emph{weak Hausdorff compactification} of $(X,\tau)$ is a triple $(\iota,Y,\sigma)$ such that:

\begin{itemize}
\item
$(Y,\sigma)$ is a compact Hausdorff space;

\item
$\iota:X\to Y$ is a continuous map; 

\item $ \iota[X]$ is dense in $Y$.
\end{itemize}
$(\iota,Y,\sigma)$ is a \emph{Hausdorff compactification} if furthermore $\iota$ is a topological embedding onto its image.
\end{definition}

\begin{definition}(Frink \cite{frink})
\label{def:normalbasis}
    Let $(X, \tau)$ be a topological space. Let $Z$ be a family of closed sets for the topology $\tau$. We say that $Z$ is: 
\begin{enumerate}[(a)]\label{def:normalbasis-weak-normal}
        
        \item \emph{weak normal} when it is a bounded distributive lattice w.r.t. finite union and intersection, and any two disjoint members $A$ and $B$ of $Z$ are subsets respectively of disjoint complements $C',D'$ of members of $Z$, i.e. $A \subseteq C'$, $B \subseteq D'$ and $C' \cap D'=\emptyset$.

\item \label{def:normalbasis-disjunctive} \emph{disjunctive} when for every $x \in X$ and for every closed set $F$ such that $x \notin F$, there exists $A \in Z$ such that $x \in A$ and $A \cap F= \emptyset$.
 \end{enumerate}

\begin{itemize}\label{def:normalbasis-weaknormalbase}
\item
A \emph{weak normal base} is a weak normal family of closed sets which is a base for the closed sets.
\item \label{def:normalbasis-weakseminormal}
$(X, \tau)$  is \emph{weak semi-normal} if it has at least one weak normal basis.

\item\label{def:normalbasis-normalbase}
A \emph{normal base} is a disjunctive and weak normal base for the closed sets.
\item \label{def:normalbasis-seminormal}
$(X, \tau)$  is \emph{semi-normal} if it has at least one normal basis. 
\end{itemize}
\end{definition}

Frink showed the following:
\begin{theorem}(Frink \cite{frink})
\label{thm:tychonoffiffseminormal}
   A $T_1$-space $(X, \tau)$ is Tychonoff if and only if it is semi-normal.  
\end{theorem}

\begin{corollary}(Frink \cite{frink})
\label{cor:normal&Frink}
    A $T_1$-space is normal Hausdorff if and only if its closed sets form a normal base.
\end{corollary}

We prove it giving a small improvement which outlines that (weak) normal bases produce (weak) Hausdorff compactifications. Again this has the added benefit of introducing ideas and lemmas that we will need in the final part containing our original results.

\begin{definition}(Frink \cite{frink})
\label{def:Zultrafilters}
    Let $Z$ be a weak normal family for a topological space $(X, \tau)$. A \emph{$Z$-ultrafilter} is a proper subset of $Z$ such that:
    \begin{enumerate}
        \item it contains every superset in $Z$ of each of its members; 
        \item it is closed under finite intersections; 
        \item it is maximal in $Z$.
    \end{enumerate}
\end{definition} 

We can define the weak Hausdorff compactification of any topological space $(X,\tau)$ induced by a weak normal family $Z$ on $X$. 

\begin{definition}(Frink \cite{frink})
\label{def:wallman-Frinkcompactif}
    Let $Z$ be a weak normal base for a $T_0$-topological space $(X,\tau)$. The \emph{weak Wallman compactification of $X$ relative to $Z$} is 
\[
\wall{Z}{X}=\bp{\mathcal{F} \subseteq Z: \mathcal{F}  \mathrm{ \ is\ a\ } Z\mathrm{-ultrafilter\ }} 
\] 
endowed with the topology $\wall{Z}{\tau}$ generated by 
\[ 
\wall{Z}{U}=\bp{\mathcal{F} \in \wall{Z}{X}: \mathrm{ \ there\ is\ some\ }A \in \mathcal{F}\mathrm{\ with\ }A \subseteq U}
\] 
    for every $U$ such that $X \setminus U \in Z$.
\end{definition}

\begin{theorem}(Frink \cite{frink})
\label{thm:W(Z,X)iscompactification}
    Let $Z$ be a (weak) normal base for a space $(X,\tau)$. Then, $(\wall{Z}{X},\wall{Z}{\tau})$ is a (weak) Hausdorff compactification of $(X,\tau)$ via the map:
\begin{align*}
    i : X \to \wall{Z}{X}  && x \mapsto i(x)=\bp{C  \in Z: x \in C}
\end{align*}
\end{theorem}

Note that a base for the closed sets of $\wall{Z}{X}$ is $\bp{\wall{Z}{A}: A \in Z}$ where \begin{equation*}
    \wall{Z}{A}=\bp{\mathcal{F} \in \wall{Z}{X}: A \in \mathcal{F}}
\end{equation*}  

It can be proved that the $0$-sets of a Tychonoff space form a normal basis whose relative Wallman compactification is clearly the Stone--\v{C}ech compactification.

In general, in any topological space the $0$-sets form a weak normal family, however if the family is not a base for the closed sets, the compact Hausdorff space it produces retains very little information from the starting one. 
For example, if the reals are endowed with the Zariski topology, the $0$-sets are given by polynomial maps  
with compact range, any such map must be constant. In particular, the unique non-empty $0$-set is the set of all reals itself. So its Wallman compactification relative to the family of $0$-sets is the one point space.

  Now, we aim to determine which conditions on a preorder $P$ and on a large family $\aaa$ of filters on $P$ ensure that  $\spa=\bp{\closed{p}: p \in P}$ forms a normal basis for $\atpa$, enabling us to construct $\wall{\spa}{\aaa}$. 
  
\begin{proposition}
\label{prop:conditionfornormalbasis} 
    Let $P$ be a bounded distributive lattice and $\aaa$ be a large family of prime filters on $P$. Then $\spa=\bp{\closed{p}: p \in P}$ is a normal basis for $\atpa$ if and only if $P$ is a normal lattice and the elements of $\aaa$ are minimal prime filters of $P$. 
\end{proposition}

\begin{proof}
\emph{}

\begin{description}[font=\normalfont\itshape] 
\item[$\spa$ is a bounded distributive lattice.] It follows from Fact \ref{fact:un&intprimefilters}. 
        \item[$\spa$ is a weak normal basis if and only if $P$ is a normal lattice.]
Suppose $\spa$ is weak normal and let $p$ and $q$  be elements of $P$ such that $ p \vee q =1_P$, i.e. $\open{p} \cup \open{q} =\aaa$. Therefore, $\closed{p} \cap \closed{q}=\emptyset$, and, thus, they are respectively subsets of  $\open{r}$ and $\open{s}$ such that $\open{r} \cap \open{s} = \emptyset$. Therefore, $r \wedge s=0_P$. Moreover, since $\closed{p}= \aaa \setminus \open{p} \subseteq \open{r}$ then $\open{p} \cup \open{r}= \aaa$, i.e. $p \vee r=1_P$. Analogously, $q \vee s=1_P$. 

The converse implication is similar.

     \item[$\spa$ is disjunctive if and only if the filters of $\aaa$ are minimal prime.] 
     Suppose $ F \in \aaa$ and $p \in F$, i.e. $F \notin \closed{p}$. Observe that requiring the existence of $\closed{q} \in \spa$ such that $F \in \closed{q}$ and $\emptyset=\closed{p} \cap \closed{q}$ is equivalent to requiring the existence of $q \notin F$ such that $p \vee q=1_P$, which, as we have already observed, is equivalent to requiring $F$ being minimal prime (by Fact \ref{fact:idealsandfilters-1}). 
    \end{description}
\end{proof}

\smallskip

The following is the key connection between the results of the previous sections and the Wallman construction:
\begin{theorem}
\label{thm:wallmancompactification-lattice}
Let $P$ be a subfit lattice and $\aaa$ be a large  family of minimal prime filters on $P$. Then, $(\wall{\spa}{\aaa},\wall{\spa}{\tau^\aaa_P} )$ is homeomorphic to $(\aaa_P, \tau^{\aaa_P}_P)$ where $\aaa_P$ is the family of \emph{all} the minimal prime filters on $P$.
\end{theorem}
 
\begin{proof}
Let 
\begin{align*}
    \phi: (\wall{\spa}{\aaa},\wall{\spa}{\tau^\aaa_P} ) &\to (\aaa_P, \tau^{\aaa_P}_P)\\
    \mathcal{F}=\bp{\closed{p}}_{p \in _{I_\mathcal{F}}}&\mapsto I_\mathcal{F}^C
\end{align*} 
\begin{description}[font=\normalfont\itshape]
        \item[$\phi$ is well defined and bijective.]
        We claim that $\mathcal{F}=\bp{\closed{p}}_{p \in I_{\mathcal{F}}}$ is a $\spa$-ultrafilter if and only if $I_{\mathcal{F}}$ is a maximal ideal on $P$. 
        
     It is evident that $\open{q} \subseteq \open{p}$ if and only if $\closed{p}\subseteq \closed{q}$, thus $I_{\mathcal{F}}$ is closed downwards with respect to $\leq$ if and only if $\mathcal{F}$ is closed upwards with respect to inclusion. Note that 
\[
\closed{p} \cap \closed{q} = \closed{p \vee q}
\]
holds (since the filters in $\mathcal{A}$ are prime): therefore $\mathcal{F}$ is closed under finite intersections if and only if $I_{\mathcal{F}}$ is closed under finite joins. 

Additionally, $\emptyset \in \mathcal{F}$ if and only if $1_P \in I_{\mathcal{F}}$.

Clearly, $\mathcal{F}$ is maximal in $\spa$ if and only if $I_{\mathcal{F}}$ is maximal in $P$.

 Thus, by Corollary \ref{cor:idealsandfilters}, we can conclude that $\phi$ is well defined and bijective.  
 
        \item[$\phi$ is  open and continuous.]
We show that 
\[
\phi[\wall{\spa}{\open{p}}]
=\mathcal{N}^{\aaa_P}_p.
\]

Recall that for $\mathcal{F}$ in $\spa$
\[
I_{\mathcal{F}}=\bp{q: \closed{q}\in\mathcal{F}}
\]
is a maximal ideal and $\phi(\mathcal{F})=I_\mathcal{F}^C$ is a minimal prime filter.

Now $\mathcal{F} \in \wall{\spa}{\open{p}}$ if and only if there is
$q \in P$ such that $\closed{q}\subseteq\open{p}$ and $\closed{q} \in \mathcal{F}$, i.e. if and only if there is $q\in I_\mathcal{F}$ and $q\vee p=1_P$, i.e. if and only if $p \in \phi(\mathcal{F})$, i.e. if and only if
$\mathcal{F} \in \mathcal{N}^{\aaa_P}_p$.

\end{description}
\end{proof}
\begin{notation}\label{not:WP}
    Let $P$ be a subfit lattice. We denote the space $(\aaa_P, \tau^{\aaa_P}_P)$, where $\aaa_P$ is the family of minimal prime filters on $P$, with $\corn{P}$.

\end{notation}

The previous results, along with Proposition \ref{prop:charsubfit}, essentially suggest that a Tychonoff space can be represented as the space generated by a large family of minimal prime filters on a subfit normal lattice. Furthermore, the relative Wallman compactification  of the space corresponds to the space formed by considering \emph{all} minimal prime filters on that lattice.

\subsection{Compact lattices}
\label{subsec:compactlattices}

Following the same pattern for the notion of subfitness, we generalize the notion of compact frame of \cite{picado2011frames} to  bounded distributive lattices.

\begin{definition}
\label{def:compactlattices} 
Let $P$ be a bounded distributive lattice.
$P$ is \emph{compact} if,  
whenever $A\subseteq P$ is such that $\sup(A)=1_P$, there exists a finite subset $B$ of $A$ such that $\sup(B)=1_P$.
 \end{definition}

\begin{fact}
\label{fact:completecomplact}
Assume $(X,\tau)$ is a compact topological space. 
Then $\tau$ is a complete, compact  lattice. Furthermore if $(X,\tau)$ is $T_1$, $\tau$ is subfit.
\end{fact}

Note that a base for a compact space which is a bounded distributive lattice need not be a compact lattice; the key point is that suprema in the base, seen as a lattice, may be computed differently than in the topology. For example:

\begin{fact}
\label{fact:completenotcompact}
Let $\bool{B}$ be an infinite complete boolean algebra. Then $\bool{B}$ is complete, subfit and not compact.

\end{fact}
\begin{proof}
Note that $\bool{B}$ is isomorphic to the the clopen set of its Stone space $\St(\bool{B})$, therefore it is a subfit lattice (and can be seen as a base for a compact Hausdorff space).
Pick a maximal antichain $\bp{b_i:i\in I}$ of $\bool{B}$ (such an antichain always exists if $\bool{B}$ is infinite, by a basic application of Zorn's Lemma).
Then 
\[
1_\bool{B}=\bigvee\bp{b_i:i\in I}
\] 
(since $\bp{b_i:i\in I}$  is a maximal antichain),
while for all finite subset $J$ of $I$ 
\[
1_\bool{B}>\bigvee \bp{b_i:i\in J}.
\]
\end{proof}

\begin{fact}
\label{fact:compactnotcomplete}
Let \( \mathcal{B} \) be the base for the interval \([0,1]\) obtained by closing under finite intersections and unions the set
\[
\bp{(a,b) \cap [0,1]: a,b \in \mathbb{Q}} 
\]
Then \( \mathcal{B} \) is a compact, subfit, but not complete, lattice.

\end{fact}

\begin{remark}
We highlight that the Wallman compactification provides a canonical way to associate to \emph{any} bounded lattice a  compact and complete lattice. Let \(P\) be an arbitrary bounded lattice. By Theorem \ref{thm:compactT1}, $\corn{P}$ is always compact and hence \(\tau_P\) is a compact complete lattice. Clearly, there exists a bounded lattice homomorphism \[h: P \to \tau_P\] given by \(h(p)=\opencorn{P}{p}.\) However, in general this is not an embedding, even if \(P\) is already complete and compact. For example, if \(P\) is the topology of the Sierpinski space, then its Wallman compactification is the space with a single point. Hence, in this case, the map \(h\) sends every non-minimum element of \(P\) to the maximum of \(\tau_P\).

Indeed, as observed in Corollary \ref{cor:subfit}, the key property that guarantees \(h\) to be an embedding is the subfitness of  \(P\).
\end{remark}



\subsubsection{On compact frames, minimal prime filters are completely prime}
Compactness is a key property of complete lattices which allows to relate (for these lattices) Cornish's notion of point (i.e. a minimal prime filter) to Isbell's notion of point (i.e. a completely prime filter). 

We begin with the following:
\begin{lemma}
\label{lemma:minprimearecompprime}
    Let \(P\) be a complete and compact lattice. Then every minimal prime filter on \(P\) is completely prime. 
\end{lemma}
\begin{proof}
Let \(F\) be a minimal prime filter on \(P\). Suppose \(\bp{p_i}_{i \in I} \) be a family of elements of \(P\) such that\[ \bigvee_{i \in I} p_i \in F.\]
Since \(F\) is minimal, there exists \(q \notin F\) such that
 \[
 q \vee \bigvee_{i \in I} p_i =1_P.
 \]
 Since $P$ is compact, there exist $i_1 \dots i_n \in I$ such that 
 \[
 q \vee \bigvee_{i = 1}^n p_i=1_P
 \] 
and hence, by the primality of \(F\), there exists \(j \in \bp{1,\dots,n}\) such that
 \[
p_{i_j} \in F.
\]
\end{proof}

We can now give a nice characterization of compact $T_1$-spaces.

\begin{proposition}  \label{lem:primecomplprime}
Assume $(X,\tau)$ is $T_1$. The following are equivalent:
\begin{enumerate}
\item \label{lem:primecomplprime-1}
$(X,\tau)$ is compact;
\item \label{lem:primecomplprime-3}
 every minimal prime filter on $\tau$ is completely prime.
\end{enumerate}
\end{proposition}

\begin{proof}
\emph{}

\begin{description}
\item[\ref{lem:primecomplprime-1} implies \ref{lem:primecomplprime-3}] 
Let $I$ be an ideal on $\tau$, then $\bp{A^c:A\in I}$ is a family of closed sets with the finite intersection property. Consequently, by compactness, there is some $x\in X$ such that
$x\not\in A$ for all $A\in I$.

We conclude that 
\[
I\subseteq\bp{U\in\tau: U\subseteq X\setminus\bp{x}}=I_x.
\]

Since $(X,\tau)$ is $T_1$, we get that $X\setminus\bp{x}\in \tau$ is the generator of $I_x$.

We have shown that all maximal ideals on a compact $T_1$-space $(X,\tau)$ are of the form $I_x$ for some $x\in X$.\footnote{There is a characterization of points valid in all $T_1$-spaces (which is the essence beneath Maruyama's results of \cite{maruyama}): the maximal and join-complete ideals of open sets are exactly those generated by the complement of a point.
In the above argument compactness is needed to identify primality with complete primality.}

Let $F$ be a minimal prime filter on $\tau$, then (by 
Corollary \ref{cor:idealsandfilters}) $F^c=\tau\setminus F$ is a maximal ideal on $\tau$, which, by what already showed, must be some $I_x$ with $x\in X$. This gives that $x\in A$ for all $A\in F$ and thus  $F_x = F$ by minimality. Therefore, $F$ is obviously completely prime.

\item[\ref{lem:primecomplprime-3} implies \ref{lem:primecomplprime-1}]
Assume $\bp{C_j:j\in J}$ is a family of closed sets for $(X,\tau)$ with the finite intersection property. Then $\bp{C_j^c:j\in J}$ (i.e. the dual of $\bp{C_j: j\in J}$) is a preideal of open sets (i.e. it does not contain a finite subcover of $X$).
By Zorn's Lemma, we can find $I$ a maximal ideal on $\tau$ extending it. Then $I$ is prime (being maximal), hence its complement $I^c=\tau\setminus I$ is a minimal prime filter on $\tau$ (by Corollary \ref{cor:idealsandfilters}), hence also completely prime (by \ref{lem:primecomplprime-3}). Now let $U=\bigcup I$; $U$ is an open subset of $X$; if $U=X$, then 
$U\in I^c$; since $I^c$ is a completely prime filter, some $A\in I$ is in $I^c$, which is clearly impossible. Hence
$U\neq X$. Now pick $x\notin U$. Then $x\in A^c$ for all $A\in I$, giving that 
$\bp{C_j:j\in J}$ has non-empty intersection as witnessed by $x$.

\end{description}

\end{proof}

\section{Adjunctions between subfit lattices and topological spaces}
\label{sec:adjunctions}
We define a contravariant reflection between compact $T_1$-spaces with closed continuous maps and distributive bounded lattices with appropriate morphisms, namely the \emph{closed subfit morphisms}, which can be restricted to a duality between compact $T_1$-spaces  and complete, compact and subfit lattices. The strongly subfit morphisms are characterized by  a first order property similar to the bounded morphisms in Esakia duality for Heyting algebras~\cite{Esakia}. Using the adjunction, we will also bring to light a lattice theoretic reformulation of the Stone-\v{C}ech compactification.

\begin{definition}
\label{def:closedmorphism}
Let $P$ and $Q$ be bounded distributive lattices. A  morphism of bounded lattices $i: P \to Q$ is a \emph{closed subfit morphism} if for every $p \in P$ and $q \in Q$ such that $ i(p) \vee q = 1_Q$ there exists $p' \in P$ such that: 
 \begin{align*}
      p \vee p'=1_P
&& i(p') \leq q
 \end{align*}
\end{definition}

\begin{remark}
Note that a surjective morphism \( i: P \to Q \) need not be closed subfit. Let \( (X, \tau) \) be any compact \( T_1 \)-space, and let \( U \) be an open subset of \( X \) that is not closed. Then the inclusion map \( U \hookrightarrow X \) induces a surjective homomorphism of bounded lattices 
\(i \colon \tau \to {\downarrow}U\) given by \(V \mapsto U \cap V.
\)
Observe that
\[
i(U) \cup \emptyset = U = 1_{{\downarrow}U}.
\]
However, if \( V \subseteq X \) is such that \( i(V) \subseteq \emptyset \), then \( V \subseteq X \setminus U \). Since \( U \) is not closed, it follows that \( V \neq X \setminus U \), and therefore \( U \cup V \neq X = 1_\tau \). This shows that \( i \) is not closed subfit.
\end{remark}

\begin{notation} $ $
     \begin{itemize}
 \item  Let $f:(X,\tau) \to (Y,\sigma)$  be a continuous map between topological spaces.
Define: \begin{align*}
    k_f: (\sigma, \subseteq) &\to (\tau, \subseteq) \\ V & \mapsto f^{-1}[V]
\end{align*}

    \item Let $i:P \to Q$ be a closed subfit morphism between bounded distributive lattices. Define: \begin{align*}
     \pi^*_i:\corn{Q} & \to\corn{P} \\ F  & \mapsto i^{-1}[F]
 \end{align*}

\end{itemize}
\end{notation}

The following propositions show that the above notation is consistent and relate the notion of closed subfit morphism to that of closed continuous map:

\begin{proposition}
\label{prop:k_fmorphism}
    Let $f: (X, \tau) \to (Y, \sigma)$ be a closed continuous map between  $T_1$-topological space. Then, $k_f: (\sigma, \subseteq) \to (\tau, \subseteq)$ is a closed  subfit morphism. 
\end{proposition}
\begin{proof}
    Obviously, $k_f$ is a morphism of bounded distributive lattices. Now, let $U \in \sigma$ and $V \in \tau$ be such that $k_f(U) \cup V = X$. Define $U'= Y \setminus f[X \setminus V]$. Since $f$ is closed, $U' \in \sigma$. Moreover,
    \begin{align*}
    k_f(U') \subseteq V : \quad    x \in k_f(U') \quad &\implies \quad x \in  f^{-1}[Y \setminus [X \setminus V]] \\ \quad &\implies \quad f(x) \notin  f[X \setminus V] \\  &\implies \quad x \notin X \setminus V \\  &\implies \quad x \in V.  
 \\ \\
    U' \cup U = Y : \quad  \quad \quad    y \notin U' \quad &\implies \quad y \notin Y \setminus f(X \setminus V)\\
        &\implies \quad y \in f[X \setminus V]  \\
         &\implies \quad y=f(x) \ \mathrm{ for \ some } \  x \in X \setminus V
          \\
          &\implies \quad y=f(x) \ \mathrm{ for \ some } \  x \notin  V \\  &\implies \quad y=f(x) \ \mathrm{ for \ some } \  x \in  k_f(U)=f^{-1}[U] \\
            &\implies \quad y \in U. 
    \end{align*}
    \end{proof}

\begin{proposition}
\label{prop:pi_map}
    Let  $P$ and $Q$ be bounded distributive  lattices and $i: P\to Q$ be a closed subfit morphism. Then $\pi^*_i: \corn{Q}\to \corn{P}$ is a closed and continuous map.
\end{proposition}

\begin{proof} $ $

    \begin{description}[font=\normalfont\itshape]
  
  \item[$\pi^*_i$ is well defined.] 
  Let $G $ be a minimal prime filter on $Q$. We leave to the reader the proof that $\pi^*_i(G)$ is a prime filter of $P$,  we only prove its minimality.  To do so, we use the characterization of minimal prime filters on a bounded distributive lattice given by Fact \ref{fact:idealsandfilters-1}. Given $p \in \pi^*_i(G) $, we must find $p'\not\in\pi^*_i(G) $ such that $p\vee p'=1_P$: since $i(p)\in G$, there exists $q \in Q$ such that $q \notin G$ and $q \vee i(p) =1_P$; since $i$ is a closed subfit morphism, there exists $p' \in P$ such that $p \vee p' =1_P$ and $i(p') \leq q$, giving that $p' \notin \pi^*_i(G)$. 

  \item[$\pi^*_i$ is continuous.] Let $p \in P$. We claim that $(\pi^*_i)^{-1}[\opencorn{P}{p}]=\opencorn{Q}{i(p)}$. 
  \begin{align*}
       (\pi^*_i)^{-1}[\opencorn{P}{p}]&=\bp{F \in \mathsf{W}(Q):\pi^*_i(F) \in \opencorn{P}{p} }=\bp{F \in \mathsf{W}(Q):i^{-1}[F] \in \opencorn{P}{p} }\\&=\bp{F \in \mathsf{W}(Q): i(p) \in F}=\opencorn{Q}{i(p)}.
  \end{align*}

 \item[$\pi^*_i$ is closed.]   Let $q \in Q$. It suffices to show that $\pi^*_i[\closedcorn{Q}{q}]$ is closed. 


We will show the following:
\begin{claim}
\label{claim-closed}
Assume $F\not\in \pi^*_i[\closedcorn{Q}{q}]$.
Then there exists $r \in P \setminus F$ such that $q \vee  i(r)=1_Q$.
\end{claim}

 Assume the claim holds, and let $F\not\in \pi^*_i[\closedcorn{Q}{q}]$ and $r$ as in the Claim. Since $i$ is a closed subfit morphism, there exists an element $p' \in P$ such that \begin{align*}
          p' \vee r =1_P  && i(p') \leq q
   \end{align*}
By the primality of $F$, we have $p' \in F$, since $r$ is not. Clearly,  $\opencorn{P}{p'} \cap \pi^*_i[\closedcorn{Q}{q}] = \emptyset$ and $F\in \opencorn{P}{p'}$. 
%
%
Since $F$ was chosen arbitrarily outside of $\pi^*_i[\closedcorn{Q}{q}]$, we have shown that any point $F$ outside $\pi^*_i[\closedcorn{Q}{q}]$ has an open neighborhood disjoint from the latter. We conclude that $\pi^*_i[\closedcorn{Q}{q}]$ is closed.

In order to complete the proof of the proposition, it remains to prove the Claim \ref{claim-closed}. 
By contraposition, the claim is equivalent to the statement: if, for all $r\in P$, $r\not\in F$  implies $q\vee i(r)\neq 1_Q$, then $F\in\pi_i^*[\closedcorn{Q}{q}]$. Further, we have that for all $r\in P$, $r\not\in F$ implies $q\vee i(r)\neq 1_Q$ if and only if $i[P\setminus F]\cup\{q\}$ generates a proper ideal of $Q$. That is, by contraposition, we need to show that if $i[P\setminus F]\cup\{q\}$ generates a proper ideal of $Q$, then $F\in\pi_i^*[\closedcorn{Q}{q}]$.

To this end, if $i[P\setminus F]\cup\{q\}$ generates a proper ideal, then, by a standard application of Zorn's Lemma, it is contained in a maximal ideal and thus it is disjoint from a minimal filter $G$ of $Q$. Now it follows that $q\not\in G$ and, for all $p\in P$, if $i(p)\in G$ then $p\not\in P\setminus F$. That is, $i^{-1}[G]\subseteq F$. But since both $F$ and $i^{-1}[G]$ are minimal filters of $P$ this means they are equal.

\end{description}

  \end{proof}

We can now define the relevant categories for which our adjunctions and dualities can be formulated.

\begin{notation}
\emph{}
\begin{itemize}
\item
\label{not:compt1}
$\bool{Cpct-T_1}_\bool{closed}$ denotes the category whose objects are compact $T_1$-spaces with arrows given by \emph{closed} continuous functions.
\item
\label{not:compt2}
$\bool{Cpct-T_2}$ denotes the full subcategory of $\bool{Cpct-T_1}_\bool{closed}$ whose objects are compact Hausdorff spaces.\footnote{Recall that a continuous map between compact Hausdorff spaces is always closed.}

\item
\label{not:los}
$\bool{DL}_\bool{closed} $ denotes the category whose objects are bounded distributive lattices with arrows given by \emph{closed subfit} morphisms. 

\item 
\label{not:cclos}
$\bool{SbfL}_\bool{closed} $ denotes the full subcategory of $\bool{DL}_\bool{closed} $ whose objects are subfit lattices.

\item 
\label{not:cclos}
$\bool{Cpl-Cpct-SbfL}_\bool{closed} $ denotes the full subcategory of $\bool{DL}_\bool{closed} $ whose objects are complete, compact and subfit lattices.

\item
\label{not:nlos}$\bool{NSbfL}$ denotes the full subcategory of $\bool{DL}_\bool{closed} $ whose objects are normal and subfit lattices.

\item 
\label{not:ccnlos}
$\bool{Cpl-Cpct-NSbfL}$ denotes the full subcategory of $\bool{NSbfL}$ whose objects are complete, compact, normal and subfit lattices. 
\end{itemize}
\end{notation}

\begin{notation}
\label{not:functors}
Let
\begin{align*}
\mathcal{L}: & \quad (\bool{Cpct-T_1}_\bool{closed})^\mathrm{op} \to \bool{DL}_\bool{closed}  \\
& \quad (X, \tau) \mapsto (\tau, \subseteq) \\
& \quad (f: (X, \tau) \to (Y, \sigma)) \mapsto (k_f: (\sigma, \subseteq) \to (\tau, \subseteq))
\end{align*}
\quad \quad \quad
\begin{align*}
\mathcal{R}: & \quad (\bool{DL}_\bool{closed} )^\mathrm{op} \to \bool{Cpct-T_1}_\bool{closed} \\
& \quad (\bool{P}, \leq) \mapsto \corn{P} \\
& \quad (i: \bool{P} \to \bool{Q}) \mapsto (\pi^*_i: \corn{Q} \to \corn{P})
\end{align*}
\end{notation}

The following is the main original contribution of the paper:

\begin{theorem}
\label{thm:adj-dual} $ $
\begin{enumerate}
\item \label{thm:adj-dual-0}
    $(\mathcal{L}, \mathcal{R})$ forms a contravariant reflection between $\bool{Cpct-T_1}_\bool{closed}$ and $\bool{DL}_\bool{closed}$; 
    \item \label{thm:adj-dual-1}
    $(\mathcal{L}, \mathcal{R})$ forms a contravariant reflection between $\bool{Cpct-T_1}_\bool{closed}$ and $\bool{SbfL}_\bool{closed}$; 
        \item \label{thm:adj-dual-2} $(\mathcal{L}, \mathcal{R})$ forms a contravariant reflection between $\bool{Cpct-T_2}$ and $\bool{NSbfL}$; 
            \item  \label{thm:adj-dual-3} $(\mathcal{L}, \mathcal{R})$ forms a duality between $\bool{Cpct-T_1}_\bool{closed}$ and $\bool{Cpl-Cpct-SbfL}_\bool{closed} $;
                    \item \label{thm:adj-dual-4} $(\mathcal{L}, \mathcal{R})$ forms a duality between $\bool{Cpct-T_2}$ and $\bool{Cpl-Cpct-NSbfL}$.
\end{enumerate}
\end{theorem}

In order to prove the theorem we need the following technical result:

\begin{proposition}
    \label{prop:counit}
    Let $P$ be a bounded distributive lattice. Then \begin{align*}
    \epsilon_P:(P, \leq) &\to (\tau_{\mathsf{W}(P)}, \subseteq)  \\
    p &\mapsto \opencorn{P}{p}
\end{align*}
is a closed subfit morphism. Moreover, if $P$ is subfit, then $\epsilon_P$ is injective and if $P$ is  subfit, complete  and compact, then $\epsilon_P$ is an isomorphism. 
\end{proposition}

\begin{proof}$ $

\begin{description}[font=\normalfont\itshape]
    \item[$\epsilon_P$ is a closed subfit morphism.] Clearly, $\epsilon_P$ is a morphism of bounded lattices. We show that $\epsilon_P$ is a closed subfit morphism. Suppose \[
\opencorn{P}{p} \cup A = \mathsf{W}(P)
\]
for some $A$ open subset of $\mathsf{W}(P)$. Since $\mathsf{W}(P)$ is compact and $A$ is a union of sets of the form $\opencorn{P}{q}$, there exists $q_1 \dots q_n \in P$ such that $A\supseteq\bigcup_{i=1}^n \opencorn{P}{q_i} $ and
\[
\opencorn{P}{p} \cup \bigcup_{i=1}^n \opencorn{P}{q_i} = \mathsf{W}(P).  
\]
Therefore, 
\begin{align*}
  p \vee \bigvee_{i=1}^n q_i = 1_P, &&\text{ and } &&\epsilon_P(\bigvee_{i=1}^n q_i) = \bigcup_{i=1}^n \opencorn{P}{q_i} \subseteq  A.
\end{align*}

\item[If $P$ is subfit, then $\epsilon_P$ is injective.]     
By Proposition \ref{prop:charsubfit} and Corollary \ref{cor:subfit}, if $P$ is subfit then $\epsilon_P$ is injective.

\item[$\epsilon_P$ is a closed subfit isomorphism if $P$ is subfit, complete and compact.]  By Lemma \ref{lemma:minprimearecompprime}, the minimal prime filters of $P$ are completely prime, which is clearly equivalent to establish the following type of identities:   
\begin{equation}\label{eqn:=minprimecompletelyprime}
\opencorn{P}{\bigvee_{i\in I} p_i} = \bigcup_{i \in I}\opencorn{P}{p_i}.
\end{equation}
 Now, define 
 \begin{align*}
     (\epsilon_P)^{-1}: (\tau_{\mathsf{W}(P)}, \subseteq) &\to (P, \leq)   \\
     \opencorn{P}{p} &\mapsto p
\end{align*}
$ (\epsilon_P)^{-1}$ is the (well-defined) inverse of $\epsilon_P$ (since every open set is of the form $\opencorn{P}{p}$ for some $p\in P$ by \ref{eqn:=minprimecompletelyprime}). Note that the bijectivity of the map $\epsilon_P$ automatically ensures that its inverse is a closed subfit
morphism.
\end{description}
\end{proof}


\begin{proposition}
    \label{prop:unit}
    Let $(X,\tau)$ be a $T_1$-topological space and \(L\) be a base for the topology that is closed under finite unions and intersections. Let
\begin{align*}
        \eta^L_X: (X,\tau) &\to \corn{L} \\
        x &\mapsto F^L_x=\bp{U \in L: x \in U}
    \end{align*}
    \begin{enumerate}
    \item \label{prop:unit-item1}\(\eta^\tau_X\) is a topological embedding;
        \item \label{prop:unit-item2} if \((X,\tau)\) is compact then \(\eta^L_X\) is a homeomorphism.
    \end{enumerate}
\end{proposition}
\begin{proof}  \( \)
\begin{description}[font=\normalfont\itshape]
\item[\ref{prop:unit-item1}] By Fact \ref{fact:omeomspacespreorders}, it is enough to show that 
\[
\bp{F^\tau_x: x  \in X} \subseteq \mathsf{W}(\tau).
\]
Clearly, $F^\tau_x$ is a prime filter. Moreover, if $U \in F^\tau_x$, then $U \cup ( X \setminus \bp{x})  =X$  and $X \setminus \bp{x} \notin F^\tau_x$. Hence, $F^\tau_x$ is minimal (note that $X \setminus \bp{x}$ is open since $(X,\tau)$ is $T_1$).

\item[\ref{prop:unit-item2}] Again, by Fact \ref{fact:omeomspacespreorders}, it is enough to show that 
\[
\bp{F^L_x: x  \in X} = \mathsf{W}(L).
\]

    \item[$\bp{F^L_x: x  \in X} \subseteq \mathsf{W}(L)$:] Clearly, $F^L_x$ is a prime filter.   Moreover, if $U \in F^L_x$, then $U \cup ( X \setminus \bp{x})  =X$.  Since \(X \setminus \bp{x}\) is an open set and \(L\) is a base, \[
    X \setminus \bp{x} = \bigcup_{i \in I} U_i
    \]
    with \(U_i \in L\) for all \(i \in I\).
    Hence, by compactness there exist \(i_1, \dots, i_n \in I\) such that \[
    U \cup \bigcup_{i=1}^n U_i=X. 
    \]
    Clearly, \( \bigcup_{i=1}^n U_i \in L\) and \( \bigcup_{i=1}^n U_i \notin F^L_x\), since \(\bigcup_{i=1}^n U_i \subseteq X \setminus \bp{x}\). Hence, \(F^L_x\) is a minimal prime filter on \(L\). 
    \item[$\bp{F^L_x: x  \in X} \supseteq \mathsf{W}(L)$:]  Suppose $G \in \mathsf{W}(L) $. If \[
\bigcap_{U \in L \setminus G}( X \setminus U )= \emptyset
\]
then, by the compactness of the space, there exists $U_1 \dots U_n \in L \setminus G$ such that
\[
\bigcap_{i=1}^n (X \setminus U_i) = \emptyset \quad \quad \mathrm{\ i.e. \ } \quad \quad \bigcup_{i=1}^n  U_i = X
\]
which is a contradiction, as $G$ is prime. Now, if \[
x \in \bigcap_{U \in L \setminus G} (X \setminus U) \]
this means that $F^L_x \subseteq G$. We conclude by minimality.
\end{description}
\end{proof}

\begin{remark}
Note that:
\begin{itemize}
\item 
The assumption that $L$ is closed under finite unions and intersections is rather restrictive: the regular open sets of a topology $\tau$ may not form such a base, as the union of two regular open sets may not be regular. Accordingly the above Lemma fails for $L=\RO(\tau)$ for many $(X,\tau)$; for example when $\tau$ is the euclidean topology on $X=[0;1]$ (in this case $\bool{W}(L)=\mathrm{St}(\RO(\tau))$ is extremally disconnected compact Hausdorff, while $X$ is connected).
\item
Given a $T_1$-space $(X,\tau)$ different bases $L$ for $\tau$ closed under finite unions and intersections will give rise to non-homeomorphic compactifications of $X$; a meaningful classification of these is an elusive problem already present in Frink's work \cite{frink} (even if Frink states it only for spaces admitting such a base which is  a normal lattice, i.e. the Tychonoff spaces). 
\end{itemize}
\end{remark}
   The proof of Theorem \ref{thm:adj-dual} is now straightforward:
    
    \begin{proof}[Proof of Theorem \ref{thm:adj-dual}] It follows immediately from 
 Proposition \ref{prop:pi_map}, Proposition \ref{prop:k_fmorphism},   Proposition \ref{prop:counit} and Proposition \ref{prop:unit}. Clearly, the counit and the unit of the adjunctions are: 
    \begin{align*}
         &\epsilon:  \bool{1}_{\bool{DL}_\bool{closed} }  \to \mathcal{L} \circ \mathcal{R^\mathrm{op}} & &\eta: \bool{1}_{\bool{Cpct-T_1}_{\bool{closed}} }\to \mathcal{R} \circ \mathcal{L^\mathrm{op}}  \\
             &\epsilon_P:(P, \leq) \to (\tau_{\mathsf{W}(P)}, \subseteq) & &\eta_X:  (X, \tau) \to \corn{\tau}
    \end{align*}         
\end{proof}
In view of the above proposition we introduce the following:

\begin{definition}
Given a $T_1$-space $(X,\sigma)$,
$\bool{W}(\sigma)$ (as defined in Notation \ref{not:WP}) is the \emph{Wallman compactification} of $(X,\sigma)$.
\end{definition}

This compactification does not have the nice universal properties of the Stone-\v{C}ech compactification (see for example this \href{https://math.stackexchange.com/questions/3892507/what-happens-to-the-stone-cech-compactification-if-you-change-compact-hausdorff/3892815#3892815}{math.stackexchange} question).

\subsection{The maximal normal sublattice of a complete, compact and subfit lattice}
\label{subsection:maximalnormalsublattice}

We move the Stone-\v{C}ech compactification theorem using our adjunction to obtain a dual result on lattices which is tracked also in \cite[Section 9]{Chandler1998}.

\begin{definition}
\label{def:stonecech}
   The \emph{weak Stone-\v{C}ech compactification} of a topological space $(X,\tau)$ is its unique weak Hausdorff compactification $(\iota, \beta X, \beta \tau)$ such that any continuous map $f : (X,\tau) \rightarrow (K, \sigma)$, where $(K, \sigma)$ is a compact Hausdorff space, extends uniquely to a continuous map \[\beta f : (\beta X, \beta \tau) \rightarrow (K, \sigma)\]
\[
\begin{tikzcd}
(X, \tau) \arrow[r, "\iota"] \arrow[d, "f"'] & (\beta X, \beta \tau) \arrow[dl, "\beta f"] \\
(K, \sigma) &
\end{tikzcd}
\]
Whn $\iota_X$ is a \emph{topological embedding}, $(\iota,\beta X, \beta \tau)$ is the \emph{Stone-\v{C}ech compactification} of $(X, \tau)$.
\end{definition}

\begin{lemma}
\label{lemma:morphism-norm-subfit}
    Let $P,Q$ be bounded distributive lattices. Suppose $Q$ is normal. Let $i: Q \to P$ be a morphism of bounded lattices. Then \begin{align*}
        f_i: \corn{P} &\to \corn{Q}  \\
        F &\mapsto \bp{q \in Q: i(q) \in F,\, i(r) \notin F \text{ and } q \vee r=1_Q \text{ for some } r \in Q}
    \end{align*}
    is a continuous closed map. Moreover, if $i$ is injective then  $ f_i$ is surjective. 
    \end{lemma}
    
\begin{proof} $ $

    \begin{description}[font=\normalfont\itshape]
    
  \item[$f_i$ is well defined.] 
    Let $F \in \mathsf{W}(P)$. Clearly, $f_i(F)$ is upward closed. Now, if $p, q \in f_i(F)$, there exist $r,s \in Q$ such that $p \vee r=q \vee s=1_Q$ and $i(r),i(s) \notin F$. Therefore, $i(r \vee s) \notin F $, by the primality of $F$ and, since $(p \wedge q) \vee r \vee s=1_Q$, we can conclude that $p \wedge q \in f_i(F)$. Clearly, $0_Q \notin f_i(F)$, hence $f_i(F)$ is a filter. 
    
    Now, we want to show that $f_i(F)$ is prime. Let $p \vee q \in f_i(F)$. Hence, let $r \in Q$ be such that $i(r) \notin F$ and $p \vee q \vee r=1_Q$. Since $F$ is prime and $i(p \vee q) \in F$, $i(p) \in F$ or $i(q) \in F$. If $i(p) \in F$ and $i(q) \notin F$, then $i(q \vee r) \notin F$, hence $p \in f_i(F)$. Similarly, if $i(q) \in F$ and $i(p) \notin F$, then $q \in f_i(F)$. Now, let $i(p),i(q) \in F$. By contradiction, assume $p,q \notin f_i(F)$. Since $Q$ is normal, there exist $s,t \in Q$ such that
    \begin{align*}
        p \vee s =1_Q && q \vee r \vee t=1_Q && s \wedge t=0_Q
    \end{align*}
    Hence $i(s)\in F$ -- otherwise, $p \in f_i(F)$. Similarly, $i(r \vee t) \in F$ and, thus, $i(t) \in F$ --since $i(r) \notin F$. However, this is absurd since $i(s) \wedge i(t)=0_P$.
    
    We are left to show minimality of $ f_i(F)$: assume $G\subseteq f_i(F)$ is a prime filter, given $q\in f_i(F)$, find $r\in Q$ with $i(r)\not\in F$ and $q\vee r=1_Q$. Then $r\not\in G$, else
$i(r)\in F$, since $G\subseteq f_i(F)$. By the primality of $G$, $q\in G$; hence $G\supseteq f_i(F)$ holds as well, and we are done.

  \item[$f_i$ is continuous.] Let $q\in Q$. We need to show  that $f_i^{-1}[\opencorn{Q}{q}]$ is open. Clearly, 
\[
f_i^{-1}[\opencorn{Q}{q}]=\bp{F \in \mathsf{W}(P) : \, i(q) \in F, i(r) \notin F \text{ and } q \vee r=1_Q \text{ for some } r \in Q}\] 

  Let $F \in f_i^{-1}[\opencorn{Q}{q}]$. Let $r$ be the witness that ensures $q$ being in $f_i(F)$. Since $Q$ is a normal lattice,  there exist $t,s \in Q$ such that \begin{align*}
        q \vee t = 1_Q && r \vee  s = 1_Q && s \wedge t =0_Q
    \end{align*}
    We claim  that $\opencorn{P}{i(s)}$ is an open neighborhood of $F$ contained in $f_i^{-1}[\opencorn{Q}{q}]$.
    
    Clearly, $F \in \opencorn{P}{i(s)} $, since $i(r) \notin F, i(r) \vee i(s)=1_P$ and $F$ is prime. Now, let $G \in \opencorn{P}{i(s)}$. Since $i(s) \in G$, $i(t)$ is not in $G$ and hence $i(q)\in G$, and $G \in f_i^{-1}[\opencorn{Q}{q}] $.

  \item[$f_i$ is closed.] It follows from the fact that every continuous map from a compact space to a compact Hausdorff space is closed.
    
\item[If $i$ is injective, then $ f_i$ is surjective.] Suppose $i$ is injective. Let $G \in \mathsf{W}(Q)$. Clearly, $i[G]$ is a prefilter on $P$. By an easy application of Zorn's Lemma, we can find a maximal  (and prime) filter $H$ on $P$ and a minimal prime filter $F$ on $P$ such that $i[G], F \subseteq H$. 
We prove by contradiction that $f_i(F)=G$: if they were different, by minimality of the filters, there exists $p \in f_i(F) \setminus G$ (else $G$ is not a minimal prime filter). Hence there exists $q \in Q$ such that $i(q) \notin F$ and $q \vee p=1_Q$. Since $Q$ is normal, there exists $s,t \in Q$ such that \begin{align*}
        p \vee t = 1_Q && q \vee  s = 1_Q && s \wedge t =0_Q
    \end{align*}
Since we are working with prime filters, $t \in G$ and $s \in f_i(F)$. Hence, $i(t), i(s) \in H$. This means that $0_P \in H$, a contradiction. 
\end{description}

\end{proof}

It is well known that the Stone-\v{C}ech compactification of a Tychonoff space is given by the space of its maximal filters of zero sets. However, it is possible to generalize this result to  arbitrary topological
  spaces, though at the expense of loosing the injectivity of the map $X \to \beta X$. For completeness, we sketch the proofs here.

\begin{definition}
\label{def:(co)zero-set}
    Let $(X, \tau)$ be a topological space. \begin{itemize}
        \item $C \subseteq X$ is a \emph{$0$-set} if there exists $f : X \to [0,1]$ continuous such that $f^{-1}[0] = C$; 
        \item $U \subseteq X$ is a \emph{cozero-set} if its complement is a $0$-set. 
    \end{itemize}
\end{definition}

\begin{lemma}
\label{lemma:cozerosetsarenormal}
    Let $(X,\tau)$ be a topological space. Then the set of cozero-sets of $X$ equipped with the inclusion is a normal distributive bounded lattice. 
\end{lemma}
\begin{proof} See for example \cite[Proposition D.3.4]{VIABOOKFORCING}.
%
%
\end{proof}

\begin{proposition}
\label{prop:stonecechaasfiltersofcozerosets}
Let $(X,\tau)$ be a topological space. Let $i: (Q, \subseteq) \to (\tau, \subseteq)$ be the inclusion of the lattice of cozero sets $Q$ of $X$ into the lattice of the open sets of $X$.   Then, $(f_i, \mathsf{W}(Q), \tau_Q)$ is the weak Stone-\v{C}ech compactification of $X$, where $f_i$ is the map defined in Lemma \ref{lemma:morphism-norm-subfit}.
\end{proposition}

\begin{proof}
Observe that by Theorem \ref{thm:wallmancompactification-lattice}, the space $\corn{Q}$ is homeomorphic to the weak compactification of $X$ given by the Wallman space given by maximal filters of $0$-sets. The homeomorphism identifies every minimal prime filter $F$ of cozero-sets for $\tau$ with the maximal filter of $0$-sets for $\tau$ given by: 
\[
\bp{X \setminus U :U\in Q \setminus F }.
\]
Moreover, the map $f_i$ sends each $x \in X$ in the minimal prime filter of cozero-sets
\[ 
f_i(x) =\bp{U \in Q: x \in U, x \notin V, U \cup V = X \text{ for some } V \in Q},
\]
which is identified (via the homeomorphism) with the maximal filter of $0$-sets given by 
\[
\bp{C \ 0\text{-set for }\tau: x \in C \text{ or } C \cap D \neq \emptyset \text{ for all } D \ 0\text{-set for }\tau \text{ such that }x \in D}.
\]
Recall that the open sets in the Wallman compactification of $(X,\tau)$ relative to the base of the $0$-sets for $\tau$ are generated by the sets 
\[
\bp{F: \text{ there exists } C \in F \text{ such that } U \supseteq C}
\]
as $U$ ranges among the cozero-sets of $X$. 
\begin{description}[font=\normalfont\itshape]
\item[$\corn{Q}$ is compact Hausdorff.] It follows from Theorem \ref{thm:compactT2}.
\item[$f_i$ is a surjective continuous and closed map.] It follows from Lemma \ref{lemma:morphism-norm-subfit} as $i$ is clearly injective and its domain is a normal lattice.  
\item[ $(\mathsf{W}(Q), \tau_Q, f_i)$ satisfies the unique extension property.] We show that any continuous $f : X \to K$ with $K$ compact Hausdorff extends uniquely to a continuous $\bar{f} : \mathsf{W}(Q) \to K$ such that $\bar{f}|_X = f$. 

We assume the reader is familiar with the theory of convergence via nets as presented for example in \cite[Chapter 1]{pedersen1989analysis} (to which we conform our notation). The proof sketch below condenses in few lines the proof of the Stone-\v{C}ech compactification theorem of  a Tychonoff space presented in \cite[Thm. D.3.13]{VIABOOKFORCING}.

Let $F$ be a maximal filter of $0$-sets of $X$. Now, let $(x_C)_{C \in F}$ be a net such that $x_C \in C$ for all $C \in F$. 
By \cite[Theorem 1.3.8]{pedersen1989analysis} there is a universal subnet $(x_\lambda)_{\lambda\in\Lambda}$ of $(x_C)_{C \in F}$, i.e. a subnet which -for any $Y\subseteq X$- is either eventually in 
 in $Y$ or eventually in $X\setminus Y$.
Let $x \in X$, then, recalling the definition of $i:Q\to\tau$ and $f_i:X\to\bool{W}(Q)$,
\[
f_i(x)=\bp{Z \text{ } 0\text{-set}: x \in Z \text{ or } Z \cap D \neq \emptyset \text{ for all } D \ 0\text{-set} \text{ such that }x \in D}.
\]
A standard argument shows that $(f_i(x_\lambda))_{\lambda\in\Lambda}$ 
 converges to $F$. Since the image of any universal net under any function is again a universal net \cite[1.3.7]{pedersen1989analysis}, and  $K$ is compact Hausdorff, the image net $(f(x_\lambda))_{\lambda\in\Lambda}$ converges to some unique point $\bar{f}(F)$ (see \cite[Prop 1.5.2, Thm 1.6.2]{pedersen1989analysis}).

Now it can be shown that if $f:X\to K$ is continuous, $\bar{f}(F)$ does not depend on the choice of the net with values in $X$ converging to $F$ (we leave to the reader to check this property). In particular, $\bar{f}$ is well defined. The uniqueness and continuity of $\bar{f}$ follows from the surjectivity of $f_i$. Again we leave to the reader the details.  
\end{description}
\end{proof}

\begin{theorem}
\label{thm:stonecech-lattices}
   Every  complete, compact and subfit lattice contains a largest normal sublattice which is also complete, compact and subfit. 
\end{theorem}

\begin{proof}
    Let $P$ be a complete, compact and subfit lattice. By Theorem \ref{thm:adj-dual}, $P \cong \tau_{\mathsf{W}(P)}$. We feel free to identify $P$ and $ \tau_{\mathsf{W}(P)}$ in the sequel.
    
\begin{description}[font=\normalfont\itshape]

\item[Existence.] Let $Q$ be the normal sublattice of the cozero-sets of $\mathsf{W}(P)$. By Proposition \ref{prop:stonecechaasfiltersofcozerosets}, $(f_i, \mathsf{W}(Q), \tau_Q)$ is the weak Stone-\v{C}ech compactification of $\corn{P}$. In particular, the map $f_i: \mathsf{W}(P) \to \beta \mathsf{W}(P)$ is a surjective, closed and continuous map. By Theorem \ref{thm:adj-dual}, $k_{f_i}: \tau_{\beta \mathsf{W}(P)} \to P$ is an injective closed subfit morphism. Hence, $\tau_{\beta \mathsf{W}(P)}$ is a normal sublattice of $P$. 
        
        \item[Maximality.]
Assume $R$ is a normal sublattice of $P$. Hence, there exists an injective morphism of bounded distributive lattices $i: R \to P$. By Lemma \ref{lemma:morphism-norm-subfit}, $f_i: \mathsf{W}(P) \to \mathsf{W}(R)$ is a continuous and closed surjective map. 
        Moreover, $\mathsf{W}(R)$ is a compact Hausdorff space, by the normality of $R$. Hence, using the universal property of the weak Stone-\v{C}ech compactification, there exists a surjective, continuous -- and automatically closed -- map $\beta f_i :\beta \mathsf{W}(P) \to \mathsf{W}(R) $. 
        
        \[
\begin{tikzcd}
\mathsf{W}(P) \arrow[r,  "\iota"] \arrow[d, "f_i",two heads] & \beta \mathsf{W}(P) \arrow[dl,  "\beta f_i",two heads] \\
\mathsf{W}(R) &
\end{tikzcd}
\]
Then, by Theorem \ref{thm:adj-dual}, $\pi^*_{f_i}:\tau_{\mathsf{W}(R)} \to \tau_{\beta \mathsf{W}(P)} $ is an injective closed subfit morphism and, hence, $R$ is isomorphic to a sublattice of $\tau_{\beta \mathsf{W}(P)}$. 
        \end{description}
\end{proof}

Summing up the Wallman compactification of a $T_1$-space always projects onto its weak Stone-\v{C}ech compactification and overlaps with it when the space is normal.

\section{Connections with other duality results and a characterization of compactness for $T_1$-spaces}
\label{sec:relation with pre-existence literature}

In this section we relate the results of the present paper to those of a similar flavour we have been able to trace in the literature. We assume throughout familiarity of the reader with Isbell's duality between sober spaces and spatial frames and Stone's duality between spectral spaces and distributive lattices, and also on how the two compare to each other. The necessary minimal details on these dualities are in any case collected in the Appendix, our reference is \cite[Chapter 6]{gehrke2024topological}. 

We first compare with care (but only at the level of objects) the results of Section \ref{sec:adjunctions} to Stone and Isbell's duality. 

Next, we bring to light the connections with two other recently introduced dualities appearing respectively in \cite{bice2020wallmandualitysemilatticesubbases,maruyama}. We obtain as side results a nice lattice theoretic characterization of compactness and continuity for the class of $T_1$-spaces.

\subsection{The canonical embedding of the Wallman compactification of a topological space into its sobrification} \label{subsec:cornishvsisbell}
Let $(X, \tau)$ be a topological space. The \emph{specialisation order} of $\tau$ is the binary relation $\leq$ on $X$ defined by
\[
    x \leq y \iff \text{for every } U \in \tau, \text{ if } x \in U, \text{ then } y \in U.
\]
Clearly, the specialisation order of a \(T_1\) topological space is trivial. 

\begin{notation}
Let  \((X,\tau)\) be a topological space. 
\begin{itemize}
    \item \label{not:min} We denote by \(\bool{min}(X)\) the subspace of \(X\) consisting of the minimal points of $X$ with respect to the specialisation order.
    \item \label{not:sober} We denote by \(\bool{Sob}(X)\) the sober space \(\bool{Pt}(\tau)\) (where $\bool{Pt}(\tau)$ consists of the completely prime filters on the spatial frame $\tau$ topologized as expected, see Def. \ref{def:isbellpoint} for details).  \(\bool{Sob}(X)\) is the \emph{sobrification} of \(X\).
\end{itemize}
     
\end{notation}

If \(L\) is a distributive lattice then the specialisation order on \(\bool{St}(L)\) (where the latter is the family of prime filters on $L$ topologized as expected, see Def. \ref{def:stonedualspace} for details) 
is given on prime filters $x,y$ on $L$ by\footnote{See Def. \ref{def:stonedualspace} for more details, for $p\in L$ $\eta(p)$ is the family of prime filters to which $p$ belongs.} \begin{align*}
    x \leq y &\iff \text{for every } p \in L, \text{ if } x \in \eta(p), \text{ then } y \in \eta(p) \\ &\iff \text{for every } p \in L,  \ p \in x, \text{ then } p \in y \\  &\iff x \subseteq y,
\end{align*} 
hence (see also for details the latter part of the Appendix, and in particular Diagram \ref{diag:st-pt})
\[
\bool{W}(L) = \bool{min}(\mathrm{St}(L)) \cong \bool{min}(\bool{Pt}(\bool{Idl}(L))),
\]
where $\bool{Idl}(L)$ is the spatial frame given by the ideals on $L$.
\begin{lemma}
    \label{lemma:ideal-cmp&top}
    Let \(L\) be a basis for a topological space \((X,\tau)\) that is closed under finite unions and intersections. Then, the inclusion \[
i : L \to \tau
\]
extends to a surjective homomorphism of frames \[
j: \bool{Idl}(L) \to \tau
\]
given by \(j(I) = \bigcup I \). 
\end{lemma}
\begin{proof}
\(j\)  obviously preserves the order and finite meets. Moreover, \( j \) extends \( i \) via the identification of \( U \) with \( \downarrow U \) for any \( U \in L \).
To show that \(j\) preserves the arbitrary joins, recall that if \(\bp{I_\alpha}_{\alpha \in A} \subseteq \bool{Idl}(L)\) then for any \(U \in L\)
\[
U \in \bigvee_{\alpha \in A} I_\alpha \iff U \subseteq U_1 \cup \dots \cup U_n 
\text{ for some } U_1 \in I_{\alpha_1}, \dots, U_n \in  I_{\alpha_n}. \] 
Hence, for any \(x \in X\)
\begin{align*}
    x \in j (\bigvee_{\alpha \in A} I_\alpha) &\iff x \in \bigcup \cp{ \bigvee_{\alpha \in A} I_\alpha} \\ &\iff x \in U \quad \text{ for some } U \subseteq U_1 \cup \dots \cup U_n   \text{ with }U_1 \in I_{\alpha_1}, \dots, U_n \in  I_{\alpha_n} 
    \\
    &\iff x \in U_\alpha \quad \text{ for some } U_\alpha \in  I_{\alpha} 
    \\
    &\iff x \in \bigcup_{\alpha \in A} \cp{ \bigcup I_\alpha} \\
     &\iff x \in \bigcup_{\alpha \in A} j(I_\alpha)
\end{align*}
Now, let \(V \in \tau\). Since \(L\) is a basis for \(\tau\)
\[
V = \bigcup_{\alpha \in A} U_\alpha
\]
for some index set $A$ with \( U_\alpha \in L\) for all \(\alpha \in A\). Hence, \[j\cp{\bigvee_{\alpha \in A} \cp{\downarrow U_\alpha}} = \bigcup_{\alpha \in A} j(\downarrow U_\alpha)=  V\]
\end{proof}

The following lemma highlights the connection between the Wallman space and the  Stone dual space of a subfit lattice.

\begin{lemma}
\label{lemma:stonedual-subfit}
   Let $L$ be a subfit lattice. Then the inclusion map 
\[ 
   \iota: \mathsf{W}(L) \to \mathrm{St}(L)
   \] 
   is a topological embedding.
Moreover: 
   \begin{enumerate}
       \item \label{lemma:stonedual-subfit-1} if $L$ is normal, then $\iota(\mathsf{W}(L))$ is a retract of  $ \mathrm{St}(L)$;
       \item \label{lemma:stonedual-subfit-2} if $L$ is a Boolean algebra, then $\iota$ is a homeomorphism.
   \end{enumerate}
\end{lemma}
\begin{proof}
Obviously, $\iota$ is a topological embedding, since the topology on $\mathsf{W}(L)$ is generated by the intersection with this space of the basic open sets generating the topology on $\mathrm{St}(L)$. 
\begin{description}[font=\normalfont \itshape]
    \item[\ref{lemma:stonedual-subfit-1}.] Suppose $L$ is normal. Let \begin{align*}
        r:  \mathrm{St}(L) &\to \mathsf{W}(L)\\
x &\mapsto r(x)   
    \end{align*}
    where for $x\in \mathrm{St}(L)$, $r(x)$ is given by
\[
    \bp{p \in x: p \vee q =1_P \mathrm{\ for \ some \ } q \notin x } 
    \]
    We leave to the readers to prove that \( r\) is well defined and continuous. The idea is to replicate the proof provided in Lemma \ref{lemma:morphism-norm-subfit} in the  case where \( i \) is the identity of $L$. Moreover, $r$ restricted to $\mathsf{W}(L)$ is the identity. 
    \item[\ref{lemma:stonedual-subfit-2}.] It follows from the fact that the prime filters on a Boolean algebra are exactly the ultrafilters on it. 
\end{description}
\end{proof}

Let \(L\) be a basis for a topological space \((X,\tau)\) closed under finite unions and intersections. Consider the surjective homomorphism of frames described in Lemma \ref{lemma:ideal-cmp&top}
\(
j: \bool{Idl}(L) \to \tau.
\)  
By applying Isbell’s duality, and by the commutativity of Diagram \ref{diag:st-pt} relating Isbell's and Stone's dualities, we obtain a continuous injective map from  
\(\bool{Sob}(X)\) to \(\bool{Pt}(\bool{Idl}(L)) = \bool{St}(L)\).  
Thus, we have the following inclusion:  
\begin{equation}
\label{eqn:subspaces}
    (X,\tau) \hookrightarrow \bool{Sob}(X) \hookrightarrow \bool{St}(L),
\end{equation}
where the first inclusion is given by the map that assigns to each point \(x \in X\) the point in \(\bool{Sob}(X)\) corresponding to the neighborhood filter of \(x\) and the second maps a completely prime filter on $\tau$ to its intersection with $L$. 

Now, suppose \(F\) is a compact frame. By Lemma \ref{lemma:minprimearecompprime}, any minimal prime filter on \(F\) is completely prime. Therefore, the minimal prime filter on \(F\) coincides with the minimal completely prime filters on \(F\).  
Hence, 
\[
\bool{W}(F)=  \bool{min}( \mathrm{St}(F))\cong \bool{min}( \bool{Pt}(F)).
\]  
Therefore, by Proposition \ref{prop:unit}, if \((X,\tau)\) is a compact  \(T_1\) space then \[
X \cong \bool{W}(\tau)  \cong  \bool{min}(\bool{Sob}(X)).\]
But, again by Proposition \ref{prop:unit}, \[
X \cong \bool{W}(L)  \cong  \bool{min}(\mathrm{St}(L))\]
for any base \(L\) that is closed under finite unions and intersections; hence
\[
\bool{min}(\bool{Sob}(X)) \cong X \cong \bool{min}(\mathrm{St}(L))\]
for any such base $L$. 
In summary, we have showed that any topological space \((X,\tau)\) can be embedded into its sobrification, which in turn can be embedded into the Stone space associated to any base \(L\) for \(\tau\) closed under finite unions and intersections. In the case of a compact space $(X,\tau)$, the three spaces $\bool{Sob}(X), \mathrm{St}(L)$, $X$ are not necessarily homeomorphic, however the subspace given by their minimal points according to the respective specialization order coincide, and is the Wallman compactification of the base $L$. Furthermore note that:
\begin{itemize}
\item 
for $L\neq L'$ distinct bases of $\tau$ both closed under finite unions and intersections, $\mathrm{St}(L)$ and $\mathrm{St}(L')$ are not homeomorphic: $L$ (respectively $L'$) is isomorphic to the compact open sets of the Stone topology on $\mathrm{St}(L)$ (respectively $\mathrm{St}(L')$), by the results of \cite[Section 6.1]{gehrke2024topological}.
\item
for   a $T_0$-space $(X,\tau)$ and  a base $L$ for $\tau$ closed under finite unions and intersections, $\bool{Sob}(X)$ may be a proper subspace of $\St(L)$. 
\end{itemize}
The examples to follow will give concrete instantiations of the above.

\subsubsection{Examples}
We conclude this first part of the Section with a few examples:
\begin{example} \( \)
    \label{ex:sober}\begin{itemize}
        \item    Let \((X,\tau)\) be \(\mathbb{N}\) equipped with the cofinite topology. It is easy to see that \(X\) is a compact \(T_1\)-space. However, it is not sober, since \(\mathbb{N}\) itself is an irreducible closed set that is not the closure of any point. As a result, \(X\) is a proper subset of its sobrification. 

The completely prime filters on \(\tau\) consist of the neighborhood filters of the elements of \(X\) along with the filter containing all non-empty open sets. Hence, \(\bool{Sob}(X)\) is isomorphic to the space \(\mathbb{N} \cup \bp{\infty}\), where the open neighborhood of \(\infty\) are of the union of \(\bp{+\infty}\) with a cofinite subsets of \(\mathbb{N}\).
Clearly, the minimal points of \(\bool{Sob}(X)\) with respect to its specialization order are precisely the natural numbers. Thus, as expected, we obtain  
\[
(X,\tau) \cong \bool{min}(\bool{Sob}(X)).
\]
\item \label{ex:stone} 

Let \((X, \tau)\) denote the closed interval \([0,1]\) equipped with the standard topology. Since \(X\) is sober, it follows that \(\bool{Sob}(X) = X\). 

Let \(L\) be the lattice generated by taking all the open intervals in \(X\) with rational endpoints and closing under finite unions and finite intersections. 

We now show that \(X\) is properly contained in \(\bool{St}(L)\). Let \(F\) be a prime filter on \(L\), and define:
\[
\bar{F} = \left\{ \bool{Cl}(U) : U \in F \right\}.
\]
This is a collection of closed subsets of \(X\) which clearly satisfies the finite intersection property. Therefore, there exists some point \(x \in X\) that belongs to all sets in \(\bar{F}\). Since \(X\) is a normal space and \(L\) forms a basis for the topology, such a point \(x\) is uniquely determined.

We aim to show that \(F^L_x \subseteq F\). 

Assume \(x \neq 0, 1\), and let \(U \in F^L_x\). Then there exist rational numbers \(a, a', b', b \in \mathbb{Q} \cap X\) such that \(a < a' < x < b' < b\), and the interval \(]a, b[ \in L\). Therefore, we have:
\[
[0,a'[ \cup ]a, b[ \cup ]b',1] = [0,1] \in F.
\]
Because \(F\) is a prime filter and \(x\) does not belong to either \([0,a']\) or \([b',1]\), it must be the case that \(]a,b[ \in F\), implying that \(U \in F\). The boundary cases \(x = 0\) or \(x = 1\) can be handled similarly.

We now analyze three cases:

\begin{enumerate}
    \item Assume \(x \notin \mathbb{Q}\). Let \(U \in F\). Then \(x \in \bool{Cl}(U)\). Since \(x\) is irrational and \(L\) is generated from intervals with rational endpoints, it follows that \(x \in U\). Hence, \(F = F^L_x\).
    
    \item Let \(q \in \mathbb{Q} \cap ]0,1[\). Then it is clear that the filters
    \begin{align*}
    F_q^- &= \left\{ U \in L \mid \text{for some } q' \in \mathbb{Q} \cap [0,1] \text{ with } q < q', \text{ we have } ]q, q'[ \subseteq U \right\}, \\
    F_q^+ &= \left\{ U \in L \mid \text{for some } q' \in \mathbb{Q} \cap [0,1] \text{ with } q' < q, \text{ we have } ]q', q[ \subseteq U \right\}
    \end{align*}
    are both prime filters. It is straightforward to prove that they are maximal prime filters. 
    We now prove that the only prime filters properly containing \(F_q^L\) are \(F_q^-, F_q^+\). Assume \(F\) is a prime filter such that \(F_q^L \subsetneq F\). Then there exists \(U \in F\) such that \(q \notin U\). Since \(F^L_q \subseteq F\), there exists \(q' \in \mathbb{Q} \cap [0,1]\) such that either \(q' < q\) or \(q < q'\), along with sets \(U_1, U_2 \in L\) disjoint from the interval \(]q', q[\) or \(]q, q'[\), respectively. Suppose \(q'< q\). The other case is similar. \(U\) can be written as \(U = U_1 \cup ]q', q[ \cup U_2\).
    By the primality of \(F\), it follows that  \(]q', q[ \in F\). Now take any \(p \in \mathbb{Q} \cap [0,1]\) with \(p < q\). If \(p \leq q'\), then \(]q', q[ \subseteq ]p, q[\), and so \(]p, q[ \in F\). If instead \(q' < p\), choose \(p' \in \mathbb{Q}\) such that \(p < p' < q\). Then
    \[
    ]q', p'[ \cup ]p, q[ = ]q', q[ \in F,
    \]
    and since \(q \notin [q', p'] \), we again conclude that \(]p, q[ \in F\), as required. Hence, \(F=F_q^+\). 
    
    \item If \(q = 0\), the argument is analogous to the previous case: the only prime filters that contain the neighbourhood filter of \(0\) with respect to \(L\) are \(F^L_0\) and \(F^+_0\). The case \(q = 1\) is entirely similar.
\end{enumerate}

Thus, we have shown that the points of the Stone space \(\bool{St}(L)\) are given by:
\[
\left\{ r \in [0,1] \mid r \notin \mathbb{Q} \right\} \cup \left\{ q^-, q^+, q \mid q \in \mathbb{Q} \cap ]0,1[ \right\} \cup \left\{ 0, 0^+, 1, 1^- \right\}.
\]

Therefore, we conclude that \(X\) is a proper subspace of the Stone dual space \(\bool{St}(L)\); however, as expected, it is homeomorphic to the subspace of minimal points of \(\bool{St}(L)\).

\end{itemize}
\end{example}
\subsection{Relating the duality of Section \ref{sec:adjunctions} with Maruyama's duality for $T_1$-spaces}

We give an alternative presentation of a duality Maruyama introduced for $T_1$-spaces (see \cite{maruyama}) and relate it to the results of Section \ref{sec:adjunctions}.
\begin{definition}\label{def:framecomplprimejoincomplete}

A \emph{coatom} $p$ of a bounded lattice $P$ is an element  such that \(p < 1_P\) and for no $x\in P$
\[ p < x < 1_P. \]
$P$ is \emph{coatomic} if for all $q< 1_P$ in $P$ there is some coatom $p\geq q$.
\end{definition}

Note that $T_1$-spaces have a coatomic topology, as the complement of a point is a coatom of it. Furthermore note that in any complete lattice the join-complete ideals are generated by the supremum of themselves, and the maximal ones are generated by a coatom.



\begin{lemma}
\label{lemma:ideal&coatom}
    Let \(I\) be a join-complete maximal ideal  on a complete lattice \(P\). Then \(I\) is generated by a coatom. 
\end{lemma}
\begin{proof}
    Let \(q=\bigvee I\). Since \(I\) is join-complete, \(q \in I\) and therefore \(I=\downarrow q\) by maximality of \(I\). Now, we want to show that \(q \) is a coatom. Suppose there exists \(r \in P\) such that \(q < r < 1_P\). Clearly, \(r \notin I\). Since \(I\) is a maximal ideal, it means that there exists \(s \in I\) such that \(s \vee r=1_P\) -- absurd, since \(s \leq q\). 
\end{proof}
It is also convenient to give another characterization of minimal prime filters thare are completely prime. 

\begin{corollary}
\label{cor:minprimcompprime&coatom}
    Let \( F \) be a minimal prime filter that is completely prime on a complete lattice \(P\). Then there exists a coatom \(q \in P\) such that  \[ F = \{r \in P: r \not \leq q \} \]
\end{corollary}

\begin{proof}
It follows from Lemma \ref{lemma:ideal&coatom} and Corollary \ref{cor:idealsandfilters}. 
\end{proof}

Let us now recall the relevant portions of Maruyama's article \cite{maruyama}.

%
%
%

\begin{definition}(Maruyama \cite[Def. 2.2]{maruyama})\label{def:m-spatial}
A frame \( L \) is \emph{m-spatial} if for any \( a, b \in L \) with \( a \not\leq b \), there is a maximal
join-complete ideal \( M \) of \( L \) such that \( a \notin M \) and \( b\in M \).
%
%

\label{def:m-homomorphism}
A frame homomorphism \( f : L \to L' \) between frames \( L \) and \( L' \) is called an \emph{m-homomorphism}
if for any maximal join-complete ideal \( M \) of \( L' \), \( f^{-1}[M] \) is a maximal join-complete ideal of \( L \).
\end{definition}

Observe that the complement of a maximal join-complete ideal is a minimal prime filter that is completely prime by Corollary \ref{cor:idealsandfilters}. Hence, requiring a frame to be m-spatial is equivalent to requiring that the family of its minimal prime filters that are completely prime is large. 

Maruyama's duality identifies an m-spatial frame $P$ with the $T_1$-space of its maximal join-complete ideals with topology generated by the sets 
\[
[p]=\bp{I: p\not\in I,\, I\text{ join-complete maximal ideal on } P}.
\] 
We recast below this duality according to a terminology which is more in line with the current paper.

\begin{notation}
\emph{}
\begin{itemize}
\item
    \label{not:t1}
    $\bool{T}_1$ denotes the category of $T_1$-spaces with \emph{arbitrary} continuous maps.
    
    \item
    \label{not:mfrm}
    $\bool{mSpFrm}$ denotes the category of m-spatial frames with m-homomorphisms.
\end{itemize}
\end{notation}

\begin{notation}
    \label{not:complety prime minimal filters}
    Let $L$ be a frame. We denote with $(M(L), \tau_L)$ the space $(\aaa, \tau^\aaa_L)$ where \[\aaa=\bp{F: F \mathrm{\ is \ a \ minimal \ prime \ filter \ that \ is \ completely \ prime}}\]
\end{notation}
\begin{notation}
\label{not:functors-maruyama}
Let

\begin{align*}
\mathcal{F}: & \quad \bool{T}_1\to \bool{mSpFrm} \\
& \quad (X, \tau) \mapsto (\tau, \subseteq) \\
& \quad (f: (X, \tau) \to (Y, \sigma)) \mapsto (k_f: (\sigma, \subseteq) \to (\tau, \subseteq)),
\\ 
\\
\mathcal{G}: & \quad \bool{mSpFrm}\to \bool{T}_1 \\
& \quad (L, \leq) \mapsto (M(L), \tau_L) \\
& \quad (i : L \to L') \mapsto (\pi^*_i:M(L') \to M(L)),
\end{align*}
where $\pi^*_i(F)=i^{-1}[F]$ for any minimal prime filter that is completely prime  $F$ on $L'$.
\end{notation}

\begin{theorem}(Maruyama \cite[Thm. 4.6]{maruyama}) 
\label{thm:Maruyama duality} $ $
    The functors $\mathcal{F}$ and $\mathcal{G}$ implement a duality between $\bool{T}_1$ and $\bool{mSpFrm}$.
\end{theorem}

We now compare this duality with the ones presented in the previous section.
Firstly, we give alternative frame theoretic characterizations of m-spatial frames and
m-homomorphisms.
\begin{definition}
    \label{def:stronglysubfit}
A bounded distributive lattice \(L\) is \emph{strongly subfit} if for every distinct \(p,q \in L\)  there exists a  coatom \(r \in L\) such that \(p \vee r=1_L\) and \(q \vee r\neq 1_L\) or viceversa. 
\end{definition}
Trivially a strongly subfit lattice is always subfit. Moreover, it is also coatomic. Indeed, for any \(p \in L\) there exists a coatom \(r \) such that \(1_L \vee r=1_L\) and \(r \vee p \neq 1_L\). Hence, \(p \leq r\).

\begin{proposition}
\label{prop:mspatial&stronglysubfit}
A frame $L$ is m-spatial if and only if it is strongly subfit.
\end{proposition}
\begin{proof} \( \)
\begin{description}
    \item[\((\Rightarrow).\)] Let \(L\) be a m-spatial frame. Let \(p,q\) distinct elements of \(L\). Suppose \(p \not \leq q\). Then, there exists a maximal join-complete ideal \(M\) such that \(p \not \in M\) and \(q \in M\). By Lemma \ref{lemma:ideal&coatom}, there exists a coatom \(r \in L\) such that \(M=\downarrow r\). Therefore, \(q \vee r = r\) and \(p \vee r\neq 1_L\), since \(p \notin M\). Hence, \(L\) is strongly subfit. 
    \item[\((\Leftarrow).\)] Let \(p \not \leq q\) be elements of \(L\). 
Since $ p \wedge q \neq p$ and \(L\) is strongly subfit, there exists a  coatom $r$ such that
\((p \wedge q) \vee r\neq  1_L \text{ and } p \vee r = 1_L
\).
Hence,
\[
q \vee r = (q \vee r) \wedge 1_L = (q \vee r) \wedge (p \vee r)
= (q \wedge p) \vee r \neq 1_L.
\]   
 By Lemma \ref{lemma:ideal&coatom}, \(\downarrow r  \) is a join-complete maximal ideal. Clearly, \(p \notin \downarrow r  \) and \(q \in \downarrow r \). Hence, \(L\) is m-spatial. 
\end{description}
\end{proof}

The previous proposition and  Maruyama's duality  allow for the following characterization of $T_1$-spaces: 
\begin{corollary}
    Let \((X,\tau)\) be a topological space. Then it is \(T_1\) if and only if \(\tau\) is a strongly subfit frame.
\end{corollary}
In particular we obtain a lattice-theoretic characterization of the $T_1$-separation property which makes no mentions of the points of a space, but only of the algebraic properties of the topology. The following however is open for us:
\begin{question}
Is every subfit and coatomic frame  strongly subfit?
\end{question}
We now turn to the task of giving a nice algebraic characterization of m-homomorphisms.

\begin{definition}
\label{def:subfit-morphism}
Let $P$ and $Q$ be subfit lattices. A  morphism of bounded lattices $i: P \to Q$ is  \emph{strongly subfit} if for every $p \in P$ and for every coatom $q \in Q$ such that $ i(p) \vee q = 1_Q$ there exists $p' \in P$ such that: \begin{align*}
    p \vee p'=1_P
\qquad \text{and} \qquad 
 i(p') \leq q.
\end{align*}
 \end{definition}

 Trivially any closed subfit morphism is a strongly subfit morphism.

 \begin{proposition}
     \label{prop:charsubfitmorphism}
A map \(i:L \to L'\) between strongly subfit frames is a m-homomorphism if and only if it is a strongly subfit frame homomorphism. 
 \end{proposition}
\begin{proof} \( \)   
        
        \begin{description}[font=\normalfont\itshape]
        \item[\((\Rightarrow).\)] We start by proving that if \(f:(X,\tau) \to (Y,\sigma)\) is a continuous map between \(T_1\)-topological spaces, then \(k_f\) is a strongly subfit morphism. Let $U \in \sigma$ and $V \in \tau$ be such that \(V\) is a coatom and  $k_f(U) \cup V = X$. Observe that the coatoms of the lattice of the open sets of a \(T_1\)-space are the complements of singletons. Thus, let \(V=X \setminus \bp{x}\) for some \(x \in X\). Clearly, \(f(x) \in U\) since \(x \in k_f(U)=f^{-1}[U]\). Define \(U'=Y \setminus \bp{f(x)}\). Then, clearly, \[
    U \cup U'=Y \quad \quad k_f(U') \subseteq V.
    \]
    Now,    let \(i:L \to L'\) be a m-homomorphism between subfit frames. By Theorem \ref{thm:Maruyama duality},   \(i\) is isomorphic to \(k_{\pi^*_i}\). Hence,  \(i\) is strongly subfit.
        \item[\((\Leftarrow).\)] Let  \(i:L \to L'\) be a strongly subfit morphism of frames. We want to show that \(i\) is a m-homomorphism. This is equivalent to show that for any \(F\) minimal prime filter that is completely prime on \(L'\), \(i^{-1}[F]\) is a minimal prime filter that is completely prime on \(L\). Firstly, we prove that \(i^{-1}[F]\) is a minimal prime filter. We mimic the proof of Proposition \ref{prop:pi_map} with some twists. We leave to the reader to prove that $ i^{-1}[F]$ is a prime filter on $L$,  we only prove its minimality. 
To do this we use again Fact \ref{fact:idealsandfilters-1}.
First we observe that (by Corollary \ref{cor:minprimcompprime&coatom} applied in $L'$) 
\(F=\bp{r \in L': r \not \leq s}\) for some \(s\) coatom in \(L'\). To apply Fact \ref{fact:idealsandfilters-1}, we must show that given $p \in  i^{-1}[F] $
we must find $p'\notin  i^{-1}[F]$ such that $p\vee p'=1_P$: since $F$ is minimal prime and $i(p)\in F$, there exists $q \in L'$ such that $q \notin F$, i.e. \(q \leq s\), and $q \vee i(p) =1_{L'}$  (by Fact \ref{fact:idealsandfilters-1} applied to $F,i(p)$); hence, \(s \vee i(p)=1_L'\); since $i$ is a strongly subfit morphism and $s$ is a coatom of $P$, there exists $p' \in L$ such that $p \vee p' =1_L$ and $i(p') \leq s$; therefore $p' \notin i^{-1}[F]$ (or \(s\) would be in \(F\)), as was to be shown. 


        Now, let us show that \(i^{-1}[F]\) is completely prime. Let \(\{p_i\}_{i \in I}\) be a family of elements of \(L\). Then \begin{align*}
            \bigvee_{i \in I} p_i \in i^{-1}[F]  &\iff i(  \bigvee_{i \in I} p_i) \in F\\
            &\iff  \bigvee_{i \in I} i(p_i) \in F\\
            &\iff i(p_{\bar{i}}) \in F \quad \mathrm{\ for \ some \ }\bar{i} \in I\\
             &\iff p_{\bar{i}} \in i^{-1}[F]  \quad \mathrm{\ for \ some \ }\bar{i} \in I,
        \end{align*}
        where the second equivalence follows from the fact that \(i\) is a morphism of frames.
        \end{description}
\end{proof}

Below we outline that compactness simplifies the lattice-theoretic characterization of a $T_1$-space.

\begin{lemma}
\label{lemma:m-spatial&subfit}\( \)
 If  \(L\) is a complete and compact lattice then it is an m-spatial frame if and only if it is strongly subfit frame if and only if it is subfit.
\end{lemma}
\begin{proof} By \ref{prop:mspatial&stronglysubfit}, a frame is  m-spatial if and only if it is strongly subfit.\\
Now, let \(L\) be an m-spatial frame.
We have already observed that this is equivalent to requiring that the family of its minimal prime filters that are completely prime is large. Hence, by Proposition \ref{prop:charsubfit},  \(L\) is subfit.  \\ Finally, suppose  \(L\)  is a complete, compact and subfit lattice. By Theorem \ref{thm:adj-dual}, it is a frame. Moreover, by Lemma \ref{lemma:minprimearecompprime}, the family of its minimal prime filters that are completely prime is exactly the family of its minimal prime filters. Hence, by Proposition \ref{prop:charsubfit} it is m-spatial. 
    
\end{proof}
The previous lemmata connect Maruyama's duality to ours.

\begin{corollary}
Maruyama's duality given by 
  \((\mathcal{F}, \mathcal{G})\) as in Thm. \ref{thm:Maruyama duality} 
restricts to a duality between \(\bool{Cpct}-\bool{T}_1\) and \(\bool{Cpl-Cpct-SbfL}\), where the forme former is given by compact $T_1$-spaces with \emph{arbitrary} continuous maps, and the latter is the category of complete compact subfit lattices with strongly subfit morphisms between them.
\end{corollary}

\begin{proof}
Since any minimal prime filter on a compact, complete and subfit lattice \(L\) is completely prime by Lemma \ref{lemma:minprimearecompprime}, \((M(L), \tau_L) = \corn{L}\). Thus, the thesis follows immediately from Proposition \ref{prop:charsubfitmorphism}, Lemma \ref{lemma:m-spatial&subfit} and Theorem \ref{thm:adj-dual}. 
\end{proof}

\subsection{Connections with other dualities}

We can also note that the restriction to compact Hausdorff spaces of Maruyama's duality coincides with Isbell's duality restricted to those spaces. This  follows from the observation that completely prime filters of normal, compact and subfit frames are exactly the minimal prime filters (note that the preceding results show that on compact $T_1$-spaces minimal prime filters are completely prime, but not the converse inclusion). To see this note that normal and subfit frames are regular (by \cite[Prop. 5.9.2]{picado2011frames}), and every completely prime filter on a regular frame is minimal prime (by \cite[Lemma 5.3]{bezhanishvili}). \\

There are also relevant connections between our results and those in \cite{bice2020wallmandualitysemilatticesubbases}; in the latter paper a functorial correspondence is defined between \emph{subbases closed under finite unions} for compact \(T_1\)-spaces (\(\cup\bool{Sub}\)) and a suitable category of semi-lattices (\(\vee\bool{Semi}\)), which -on the objects- is defined by a minor variation of our notion of subfitness. 
However, a substantial distinction between \cite{bice2020wallmandualitysemilatticesubbases} and the present results lies in the definition of arrows: in \cite[Def. 8.4]{bice2020wallmandualitysemilatticesubbases}, arrows between semi-lattices in \(\vee\bool{Semi}\) are defined as particular types of relations called \emph{\(\vee\)-relations} \cite[Def. 8.4]{bice2020wallmandualitysemilatticesubbases}, while arrows in \(\cup\)-subbases correspond to partial functions on the corresponding topological spaces that are \textit{very continuous} \cite[Def. 8.1]{bice2020wallmandualitysemilatticesubbases}. 
When restricted to complete, compact, subfit lattices, the two notions of arrow coincide, more precisely:  
by  Theorem \ref{thm:adj-dual}, \cite[Thm. 8.9]{bice2020wallmandualitysemilatticesubbases}, and \cite[Thm. 8.6]{bice2020wallmandualitysemilatticesubbases}, the closed subfit morphisms between complete, compact, and subfit lattices are exactly the \(\vee\)-relations between them as in \cite[Def. 8.4]{bice2020wallmandualitysemilatticesubbases} (i.e the arrows of \(\vee\bool{Semi}\)), and similarly the very continuous functions map between compact $T_1$-spaces according to the the maximal $\cup$-subbase given by the topology are the closed continuous maps, furthermore on this classes of spaces/lattices the duality presented in Thm. \ref{thm:adj-dual}(\ref{thm:adj-dual-3})  is exactly the one given in \cite[Thm. 7.9]{bice2020wallmandualitysemilatticesubbases} for objects and \cite[Sec. 8]{bice2020wallmandualitysemilatticesubbases} for arrows.

\section{Appendix: Recap on Isbell's duality for spatial frames and Stone's duality for distributive lattices}\label{Appendix}
We collect her the results and terminology we usesd in Section \ref{sec:relation with pre-existence literature} to compare the duality presented in Section \ref{sec:adjunctions} to Isbell's and Stone dualities.
We follow closely \cite{gehrke2024topological}.

Recall the following definitions:
\begin{itemize}

\item A \emph{frame} $(L,\wedge,\vee)$ is a complete distributive lattice such that the infinitary distributive law
\[
a\wedge\bigvee A=\bigvee\bp{a\wedge b:b\in A}
\]
holds for all $a\in L$ and $A\subseteq L$.
\item A map between frames is a \emph{frame homomorphism} if it preserves finite meets and arbitrary joins.
\item  An element \(m\) of a frame \(L\) is  \emph{meet-irreducible} if \(m= \bigwedge S\) with \(S \subseteq L\) finite implies that \(m \in S\).
 \item  An element $k$ of a frame $L $ is \emph{compact} if, for every directed subset $S \subseteq L$:
\[
    k \leq \bigvee S \implies \exists s \in S, \quad k \leq s.
\]
\item
A filter $F$ on a frame $L$ is \emph{completely prime} if for all $A\subseteq L$ with $\bigvee A\in F$, there is some $a\in A\cap F$. Its dual is a \emph{join-complete} prime ideal.
\item
A frame $L$ is \emph{spatial} if for all \(a,b \in L\) if \(a \not \leq b\), then there exists a completely prime filter \(F\) on \(L\) such that \(a \in F\) and \(b \notin F\). 
\item 
A frame $L$ is \emph{coherent} if its compact elements form a sublattice which generates $L$ by directed joins. 

\item A topological space $(X,\tau)$ is  \emph{sober} if  every irreducible closed subset of $X$ is the closure of a unique point. 
   
    \end{itemize}

Observe that if $L$ is a frame and $\mathcal{A}$ is the family of its completely prime filters, then requiring that $\mathcal{A}$ is large for $L$ is equivalent to requiring that $L$ is spatial. \\

\begin{notation}
    \label{not:stonedualspace}
    Let $L$ be a distributive bounded lattice.  
\begin{itemize}
    \item $\mathrm{Idl(L)}$ is the set of the ideals on \(L\);
    \item $\mathrm{PrFilt}(L)$ is the set of the prime filters on $L$;
    \item $\mathrm{PrIdl}(L)$  is the set of the prime ideals on $L$; 
    \item $\mathrm{Hom_{DL}}(L, \textbf{2})$ is the set of the bounded distributive lattice homomorphisms from $L$ to $\textbf{2}$. 
\end{itemize}
\end{notation}

\subsubsection{Stone's duality for distributive lattices}

\begin{definition}(Gehrke, van Gool \cite[Def 6.3]{gehrke2024topological})
\label{def:stonedualspace}
 Let \( L \) be a distributive lattice. The \emph{Stone dual space} of \( L \) or \emph{spectral space} of \( L \), denoted \( \mathrm{St}(L) \), is a topological space \( (X, \tau) \), where \( X \) comes with bijections $ F_{(-)}: X \to \mathrm{PrFilt}(L),  I_{(-)}: X \to \mathrm{PrIdl}(L) ,  h_{(-)}: X \to \mathrm{Hom_{DL}}(L, \textbf{2}) $ so that for all \( x \in X \) and all \( p\in L \), we have
\[
p \in F_x \iff p \notin I_x \iff h_x(p) = \top
\]
and the topology \( \tau \) is generated by the sets
\[
\eta(p) = \{ x \in X : p \in F_x \} = \{ x \in X : p \notin I_x \} = \{ x \in X : h_x(p) = \top \},
\]
where \( p \) ranges over the elements of \( L \).
\end{definition}

\begin{notation}
\label{not:DL-Spec}
\emph{}
\begin{itemize}
\item
\label{not:DL}
$\bool{DL}$ denotes the category whose objects are distributive bounded lattices with arrows given by the homorphisms between them.
\item
\label{not:Spec}
$\bool{Spec}$ denotes the category whose objects are spectral spaces with arrows given by spectral maps.\footnote{For the definitions of spectral spaces and spectral maps, see \cite[Def 6.2]{gehrke2024topological}.}
\end{itemize}
\end{notation}

\cite[Thm. 6.4, Thm. 5.38]{gehrke2024topological} shows that the correspondence between distributive bounded lattices and their Stone
dual spaces is functorial. In other words, we can define a functor
\[\mathrm{St}: \bool{DL}^{\bool{op}} \to \bool{Spec}\] that assigns to each $L$ its Stone dual space and to every 
homomorphism $h: L \to M$  the function $f: X_M \to X_L$ given by
requiring that
\[
F_{f(x)} = h^{-1}[F_x], \quad \text{for every } x \in X_M.
\]
Moreover, this functor is actually an equivalence, where the inverse is the functor
\[\mathcal{KO}: \bool{Spec}^{\bool{op}} \to \bool{DL}\] that assigns to each $(X,\tau)$ the lattice of its compact-open sets equipped with the inclusion and to every 
spectral map $f: (X,\tau) \to (Y,\sigma)$  the homomorphism $h: \mathcal{KO}(Y) \to \mathcal{KO}(X)$ given by \[
h(U)=f^{-1}[U].
\]

\subsubsection{Isbell's duality}

Now, let us recall Isbell's duality (also known as Omega-point duality) between the category of spatial frames and the category of sober spaces.

\begin{notation}
\label{not:pt(L)}
Let \(L\) be a frame. 
\begin{itemize}
   \item $\mathrm{CompPrFilt}(L)$ is the set of the completely prime filters on $L$;
   \item $\mathcal{M}(L)$ denotes the poset of the meet-irreducible elements of $L$; 
   \item $\mathrm{Hom_{Frm}}(L, \textbf{2})$ is the set of the frame homomorphisms from $L$ to $\textbf{2}$; 
   \item $\mathrm{K}(L)$ is the set of the compact elements of $L$.
   
\end{itemize}
\end{notation}

\begin{definition}(Gehrke, van Gool \cite[Def 6.12]{gehrke2024topological}) \label{def:isbellpoint}
       Let $L$ be a frame. We denote by $\bool{Pt}(L)$ a topological space which is determined up to homeomorphism by the fact that it comes with three bijections:
\begin{itemize}
    \item $F_{(-)}: \bool{Pt}(L) \to \mathrm{CompPrFilt}(L)$,
    \item $m: \bool{Pt}(L) \to \mathcal{M}(L)$,
    \item $f_{(-)}: \bool{Pt}(L) \to \mathrm{Hom_ {Frm}}(L,  \textbf{2})$,
\end{itemize}
satisfying, for any $x \in \bool{Pt}(L)$ and $a \in L$,
\[
a \in F_x \iff a \not \leq m(x) \iff f_x(a) = 1.
\]
The opens of $\bool{Pt}(L)$ are defined as:
\[
\hat{a} = \{ x \in \bool{Pt}(X) \mid a \in F_x \}.
\]
\end{definition}

\begin{notation}
\label{not:Frm-Top}
\emph{}
\begin{itemize}
\item 
\label{not:Frm}
    $\bool{Frm}$ denotes the category of frames with  frame homomorphisms.
    \item 
\label{not:SpFrm}
    $\bool{SpFrm}$ denotes the full subcategory of \(\bool{Frm}\) whose objects are spatial frames.
    \item
    \label{not:CohFrm}
    $\bool{CohFrm}$ denotes the subcategory of \(\bool{Frm}\) whose objects are coherent frames with coherent maps \footnote{For the definition of coherent map see \cite{gehrke2024topological}}.
\item
\label{not:Top}
$\bool{Top}$ denotes the category whose objects are topological spaces with continuous maps.
    \item 
\label{not:Sob}
    $\bool{Sob}$ denotes the full subcategory of \(\bool{Top}\) whose objects are sober spaces.
\end{itemize}
\end{notation}

Let
    \[ \bool{Pt}: \bool{Frm}^{\bool{op}} \to \bool{Top} \]
    be the functor that assigns to a frame \(L\) its space \(\bool{Pt(L)}\) and to a frame homomorphism \( h: L \to M \) the continuous map \( \bool{Pt}(h): \bool{Pt}(M) \to \bool{Pt}(L) \) given by
    \[
    \bool{Pt}(h)(x) = y \iff h^{-1}(F_y) = F_x \iff h^{-1}(\downarrow m(y)) = \downarrow m(x) \iff f_y = f_x \circ h.
    \]

On the other hand, let 
 \[ \Omega: \bool{Top}^{\bool{op}} \to \bool{Frm} \]
be the functor that assigns to a topological space \(X\) the frame of its open sets \(\Omega(X)\) and to a continuous map \( f : X \to Y \) between 
 topological spaces the homomorphism of frames 
\[
\Omega(f): \Omega(Y) \to \Omega(X)
\]
given by \( U \to f^{-1}[U] \).

\begin{theorem} (Gehrke, van Gool \cite[Thm. 6.13, Thm. 6.19]{gehrke2024topological})
    \label{thm:omega-ptduality}
    The functors \( (\bool{Pt}, \Omega) \) form a contravariant adjunction between \(\bool{Frm}\) and \(\bool{Top}\), which restricts to a duality between \(\bool{SpFrm}\) and \(\bool{Sob}\).
\end{theorem}

A distributive lattice $L$ is isomorphic to the lattice of compact elements of its ideal completion, i.e., \(\mathrm{K}(\mathrm{Idl}(L)) \cong L\). Moreover, it follows straightforwardly from the definition that \(F \cong \bool{Idl}(\mathrm{K}(L))\) for any coherent frame \(F\). Therefore, the following diagram  is commutative \emph{on the objects}:

\begin{equation}
\label{diag:st-pt}
\begin{tikzcd}
    \bool{Sob} \supseteq \bool{Spec} 
    \arrow[r, leftrightarrow, "\mathcal{KO}\text{-St}"] 
    \arrow[rd,leftrightarrow,  "\Omega\bool{-Pt}"'] 
    & \bool{DL} 
    \arrow[d,leftrightarrow,  "\bool{Idl-}\mathcal{K}"] \\
    & \bool{CohFrm \subseteq SpFrm}
\end{tikzcd}
\end{equation}

\bibliographystyle{plain}
\bibliography{biblio}

\begin{thebibliography}{10}

\bibitem{bezhanishvili}
G.~Bezhanishvili, David Gabelaia, and Mamuka Jibladze.
\newblock Spectra of compact regular frames.
\newblock {\em Theory and Applications of Categories}, 31:365--383, 01 2016.

\bibitem{bice2020wallmandualitysemilatticesubbases}
Tristan Bice and Wies{\l}aw Kubi\'s.
\newblock Wallman duality for semilattice subbases.
\newblock {\em Houston J. Math.}, 48(1):1--31, 2022.

\bibitem{Chandler1998}
Richard~E. Chandler and Gary~D. Faulkner.
\newblock {\em Hausdorff Compactifications: A Retrospective}, pages 631--667.
\newblock Springer Netherlands, Dordrecht, 1998.

\bibitem{Cornish1972NormalL}
William~H. Cornish.
\newblock Normal lattices.
\newblock {\em Journal of the Australian Mathematical Society}, 14:200 -- 215, 1972.

\bibitem{Esakia}
Leo Esakia.
\newblock {\em Heyting Algebras: Duality Theory}.
\newblock Springer Cham, 2019.
\newblock Translation of 1974 book. Translated by {A}nton {E}vseev.

\bibitem{frink}
Orrin Frink.
\newblock Compactifications and semi-normal spaces.
\newblock {\em American Journal of Mathematics}, 86(3):602--607, 1964.

\bibitem{gehrke2024topological}
M.~Gehrke and S.~van Gool.
\newblock {\em Topological Duality for Distributive Lattices: Theory and Applications}.
\newblock Cambridge Tracts in Theoretical Computer Science. Cambridge University Press, 2024.

\bibitem{Isbell}
John~R. Isbell.
\newblock Atomless parts of spaces.
\newblock {\em Mathematica Scandinavica}, 31:5--32, 1972.

\bibitem{MacLMoer}
Saunders Mac~Lane and Ieke Moerdijk.
\newblock {\em Sheaves in geometry and logic. A first introduction to topos theory}.
\newblock Universitext. Springer-Verlag, New York, 1994.
\newblock Corrected reprint of the 1992 edition.

\bibitem{maruyama}
Yoshihiro Maruyama.
\newblock Topological duality via maximal spectrum functor.
\newblock {\em Communications in Algebra}, 48:1--8, 02 2020.

\bibitem{pedersen1989analysis}
G.K. Pedersen.
\newblock {\em Analysis Now}.
\newblock Graduate texts in mathematics. Springer-Verlag, 1989.

\bibitem{picado2011frames}
Jorge Picado and Ale{\v{s}} Pultr.
\newblock {\em Frames and Locales: topology without points}.
\newblock Springer Science \& Business Media, 2011.

\bibitem{VIABOOKFORCING}
M.~Viale.
\newblock {\em The Forcing Method in Set Theory}.
\newblock UNITEXT. Springer Cham, first edition, 2024.

\bibitem{WallmanComp37}
Henry Wallman.
\newblock Lattices and topological spaces.
\newblock {\em Ann. of Math. (2)}, 39(1):112--126, 1938.

\end{thebibliography}

\end{document}